\documentclass[12pt]{amsart}
\usepackage{eurosym}
\usepackage{mathrsfs}
\usepackage{amsmath}
\usepackage{amsfonts}
\usepackage{amssymb}
\usepackage{amscd}
\usepackage[colorlinks]{hyperref}
\usepackage{graphicx}
\usepackage{multicol}

\usepackage[top=2cm, bottom=2cm, left=2cm, right=2.5cm]{geometry}

\newtheorem{theorem}{Theorem}

\newtheorem{proposition}[theorem]{Proposition}

\newtheorem{lemma}[theorem]{Lemma}
\numberwithin{theorem}{section}
\numberwithin{equation}{section}

\newcommand{\esp}{\mathbf{L}}

\newcommand{\R}{\mathbb{R}}

\newcommand{\h}{\mathbf{H}}

\newcommand{\Rm}{\mathbb R^+}

\begin{document}
\title{On the Benjamin Ono equation in the half line}
\author[D. Cardona]{Duv\'an Cardona}
\address{
  Duv\'an Cardona:
  \endgraf
  Department of Mathematics: Analysis, Logic and Discrete Mathematics
  \endgraf
  Ghent University, Belgium
  \endgraf
  {\it E-mail address} {\rm duvanc306@gmail.com, duvan.cardonasanchez@ugent.be}
  }
  
  \author[L. Esquivel]{Liliana Esquivel}
\address{
 Liliana Esquivel:
  \endgraf
 Gran Sasso Science Institute, CP 67100
  \endgraf 
  L' Aquila, Italia.
  \endgraf
  {\it E-mail address} {\rm liliana.esquivel@gssi.it}
  }


\begin{abstract} We consider the inhomogeneous Dirichlet initial boundary value problem for the Benjamin-Ono equation formulated on the half line. We study the global in time existence of solutions to the initial–boundary value problem.   This work  is a continuation of  the ones \cite{hayashi2010benjamin,Neumann} by Hayashi and Kaikina where the global in time existence  and the asymptotic behaviour of solutions for large time were considered.
\end{abstract}
\maketitle

\tableofcontents
\allowdisplaybreaks

\section{Introduction}
\subsection{Outline} When studying dispersive problems on the half line, boundary terms provide a perturbation to the behaviour of their solutions and the analysis for these problems can be treated by using techniques of   complex analysis, namely,  methods of analytic continuation.  By following this philosophy, and by exploiting the Calder\'on commutator technique,
in this work we shall study the initial-boundary value problem (IBV problem) for the Generalized Dispersive Benjamin-Ono equation on the half line
\begin{equation} \textnormal{(BO)}:\label{nolineal} \begin{cases}u_{t}+\mathcal{H}u_{xx}+u\partial_{x}u=0,\ t>0, \ x>0, \\
u(x,0)=u_{0}(x),\\
u(0,t)=h(t),\ t > 0, & \textnormal{ } \end{cases}
\end{equation}
where $\mathcal H$ denotes the Hilbert transform\footnote{The Hilbert transform has a dispersive effect in the BO equation. Note that  the Hilbert  transform is a definite integral, which differs from an integral  term that appears in the model equation for  shallow-water waves (see e.g. Hirota and Satsuma\cite{HS}), and it makes the properties of  these solutions very different from those of  the well-known KdV type.}, which, as a (Calder\'on-Zygmund)  singular integral operator is defined via
\begin{equation}
    \mathcal{H}u(x):=\textnormal{PV}\int\limits_{0}^\infty\frac{u(y)}{y-x}dy,\,\,u\in C^{\infty}_0(\mathbf R^+).
\end{equation}
For a comprehensive study of this operator on half-line we refer to the interested reader to \cite{E Shamir}
Between the family of dispersive equations, the Benjamin-Ono equation is a  model describing long internal gravity waves in a stratified fluid with infinite depth (see for instance,  Benjamin \cite{Benjamin} and Ono \cite{Ono}) and  turns out to be important in other physical phenomena as well (we refer the reader to Danov and  Ruderman \cite{DaR},  Ishimori \cite{Is2} and Matsuno  Kaup \cite{MK} and references therein). In the case of the whole line, for the Benjamin-Ono model are well known very noticeable properties: it defines a Hamiltonian system, can be solved by an analogue of the inverse scattering method (see Ablowitz and  Fokas\cite{AF}), admits (multi-)soliton solutions, and satisfies infinitely many conserved quantities (see Case \cite{Ca}).
 It is worth to mention that, regarding the IVP associated to the BO equation (in the case of the whole line), local and global results have been obtained by various authors. Iorio \cite{Ioro} showed local well-posedness for data in $\h^s(\mathbb{R}),$ $s>\frac{3}{2}$, and making use of the conserved quantities he extended globally the result in $\h^s(\mathbb{R}),\,s\geq 2.$ In Ponce \cite{Po}, the author extended the local result for data in $\h^{\frac{3}{2}}(\mathbb{R})$ and the global result for any solution in $\h^s(\mathbb{R}),\,\,s\geq \frac{3}{2},$ and further improvements were done by Molinet, Saut, and Tzvetkov \cite{MST3}\footnote{Showing that the Picard iteration process cannot be carry out to prove local results for the BO equation in $\h^s(\mathbb{R}).$}, Koch and Tzvetkov\footnote{Who established a local result for data in $H^s(\mathbb{R}),\,s>\frac{5}{4},$ improving the one given in \cite{Po}. }, Jinibre and Velo\cite{JV},  Tao who showed in  \cite{To} that the IVP associated to the BO equation is globally well-posed in $\h^1(\mathbb{R}),$ Ionescu and Kenig \cite{IK1}, Molinet and Riboud \cite{MR1,MR2} and \cite{KPV11}, just to mention a few. 

\subsection{Benjamin Ono Equation in the Half-line}
In the case of the half-line, the BO equation and other non-linear evolution problems for pseudo-differential operators on the half-line have been considered by Esquivel, Hayashi and Kaikina \cite{esquivel2019inhomogeneou, esquivel2018}, Hayashi and Kaikina \cite{hayashi2010benjamin,Neumann,HK2004} and Kaikina \cite{Kaikina2005,Kaikina2007,kaikina2015capillary}. In the homogeneous case of  \eqref{nolineal}, with $h(t)\equiv 0,$ or when \eqref{nolineal} is endowed with the  Neumann condition $u_{x}(0,t)=0,$  it was proved, among other things, in  \cite{hayashi2010benjamin,Neumann}, the well-posedness for  \eqref{nolineal} if $\psi\in \mathbf L^{1,a}(\mathbf R^+)\cap \textbf H^{1}(\mathbf R^+) ,$ where $a\in (0,1).$ As far as we know, the case of nonhomogenous boundary condition for the initial boundary value problem \eqref{nolineal} was not studied
previously. The main problem addressed in this paper is to study the global in time existence of solutions to  \eqref{nolineal} in the case where the initial data $\psi$ belongs to $ \textbf H^{1+\epsilon}(\mathbf R^+)\cap  \mathbf L^{1,2}(\mathbf R^+)$ and the boundary condition $h(t)\in  \mathbf H^1(\mathbf R^+)\cap \mathbf{L}^1(\mathbf R^+). $ Under these conditions, in Theorem \ref{Main result} we prove that there exist an unique global solution $u$ of \eqref{nolineal} in the space
$ \mathbf{C}([0,\infty):\mathbf H^1(\mathbf R^+)\cap \mathbf L^{2,1}(\mathbf R^+)).$ In this result we observe the influence of the boundary data on
the behavior of solutions.

The novelty of the present work is that we combine different approaches between the real and the complex analysis. We start our work by applying the analytic continuation method by Hayashi and Kaikina (in the aforementioned references) related to the Riemann-Hilbert problem. Indeed, the construction of the Green operator is based on the introduction of a suitable necessary condition at
the singularity points of the symbol, the integral representation for the sectionally analytic function, and the theory of
singular integrodifferential equations with Hilbert kernels and with discontinues coefficients, (see \cite{Kaikina2007}, \cite{Gakov} and Section \ref{Sketch} for details).

Later on, via the contraction principle, in  Theorem \ref{Main result} we deduce of global existence of a solutions $u\in \mathbf H^1(\mathbf R^+)$ to \eqref{nolineal}. Finally, via an   analysis based also, in the Calder\'on commutator technique as developed by Ponce and Fonseca \cite{FG}, we prove that $u\in \mathbf{L}^{2,1}(\mathbf R^+)$.

This paper is organised as follows. In Section \ref{Notation} we present the notation used in our work and our main result in the form of Theorem \ref{Main result}.  For the benefit of the reader, in Section \ref{Sketch} we explain the techniques that we follow in the proof of our main theorem.  The linear problem associated to \eqref{nolineal} will be analysed in Section \ref{Lproblem} and some technical lemmata will be established in Section \ref{LemmasTecnicos}. We end our work with the proof of Theorem \ref{Main result} in Section \ref{TheProofofLili}. 

\section{Notation and Main result}\label{Notation}
\noindent To state our results precisely we introduce notations and function
spaces. Let $t\in \mathbb{R}.$ We denote by $\langle  t \rangle:=\sqrt{1+t^2}$ and  $\{t\}:=\frac{|t|}{\langle t \rangle }$. We write $A\lesssim B$ if there exist a constant $C,$ such that $C$ does not depend on fundamental quantities on $A$ and $B,$ such that $A\leq CB$.

We denote the usual direct and inverse Laplace transformation  by $%
\mathcal{L}$ and $\mathcal{L}^{-1}$%
\begin{equation*}
\mathcal{L}\phi (\xi)\equiv\widehat{\phi }\left( \xi \right) :=\int\limits_{{0}%
}^{\infty }e^{-x\xi}\phi \left( x\right) dx,\quad \mathcal{L}^{-1}\phi(x) =\frac{1}{2\pi
i}\int\limits_{i\mathbf{R}}e^{ix\xi }\widehat{\phi }\left( \xi \right) d\xi .
\end{equation*}
The Fourier transform $\mathcal{F}$ is defined as 
$$\mathcal{F}(\phi)(p):=\frac{1}{\sqrt{2\pi}}\int_{-\infty}^{\infty}e^{-ix p}\phi(x)dx.$$
Let $\Omega\subset \mathbb{R}$ be an interval, for $s\geq 0 $, we define
 $$\h^s(\Omega):=\{f=F|_{\Omega}:F\in \h^s(\R)\} \ \ \text{ and } \|f\|_{\h^s(\Omega)}=\inf \|F\|_{\h^{s}(\R)}. $$  Similarly, we have $$\dot{\h^s}(\Omega):=\{f=F|_{\Rm}:F\in \dot \h^s(\R)\} \ \ \text{ and } \|f\|_{\dot \h^s(\Omega)}=\inf\limits_{F} \|F\|_{\dot \h^{s}(\R)}. $$
At this point it is important to  recall  the space $\h^s_0(\R)$, which is the closure of the class of functions in $\h^s(\mathbb{R})$ whose support lies in $\Rm.$
Recall that $\h^s(\R)$ is the set of distributions $f$ satisfying $(1+|\xi|)^s\hat{f}(\xi)\in \esp^2_\xi$, where $\hat{f}$ denotes the Fourier transform in the $x$-variables. The space
$\dot{\h^s}(\R)$  is the homogeneous analogue consisting of distributions satisfying  $|\xi|^s\hat{f}(\xi)\in \esp^2_\xi$.  

On the other hand, if we define $$\mathbf W^{2,k}(\Omega)=\left\{f: \frac{d^mf}{dx^m}\in  \mathbf{L}^2(\Omega), 0\leq m\leq k \right\},$$ with the norm $$\|f\|_{\mathbf W^{2,k}(\Omega)}=\left( \sum_{0\leq m\leq k} \left\|\frac{d^mf}{dx^m} f\right\|_{\mathbf{L}^2(\Omega)}\right),$$ via Calderon's extension theorem (Theorem 12 in \cite{calderon1961lebesgue}) and  standard Calder\'on-Zygmund estimates one can deduce  $\h^k(\mathbb R)= \mathbf W^{2,k}(\mathbb R)$, for all $k$ non-negative integer, as consequence $\h^k(\Omega)= \mathbf W^{2,k}(\Omega)$.

For $s\in \mathbb R,$ we define $\mathbf Z^{s,r}(\Omega):=\mathbf H^s(\Omega)\cap\mathbf  {{ L}}^{2,r}(\Omega)$, here $\mathbf { L}^{2,r}(\Omega)$ is the  weighted Lebesgue space 
$$\mathbf{ L}^{2,r}(\Omega)=\{\phi:\|\phi\|_{\mathbf{ L}^{2,r}}:=\| |\cdot|^{r} \phi(\cdot)\|_{\mathbf L^2}<\infty\}.$$

Now we state our main result.

\begin{theorem}\label{Main result}
For $\psi\in \mathbf Z^{1+\epsilon,1}(\mathbb R^+)$ and $h\in \mathbf Z^{1,1} $, there exist a unique global solution 
$$u\in \mathbf{C}([0,\infty):\mathbf H^1(\mathbf R^+)\cap \mathbf L^{2,1}(\mathbf R^+)),$$
such that 
$$\sup_{t>0}\|u(t)\|_{\mathbf H^1}\lesssim M_1(\|u_0\|_{\mathbf{Z}},\|h\|_{\mathbf Y}),$$
$$\|u(t)\|_{\mathbf L^{2,1}}\lesssim M_2e^{M_3t}(\|u_0\|_{\mathbf{Z}},\|h\|_{\mathbf Y}).$$
\end{theorem}
\section{Sketch of the proof}\label{Sketch}

For the convenience of the reader we briefly explain our strategy.
First of all, we consider the linear Benjamin-Ono equation with inhomogenous boundary condition
\begin{equation} \textnormal{(LBO)}:\label{lineal} \begin{cases}u_{t}+\mathcal{H}u_{xx}=0,\ t>0, \ x>0,\\
u(x,0)=\psi(x), \ x>0\\
u(0,t)=h(t),\ t > 0. & \textnormal{ } \end{cases}
\end{equation}
In Lemma \ref{propo lineal} we construct the
Green function and the Boundary operator for  Eq. \eqref{lineal}, indeed we prove that the solution  of  this equation can be represented as
\begin{equation*}
u=\mathcal{G}\left( t\right) \psi+\mathcal{B}(t)h,
\end{equation*}
where 
\begin{equation}
 \mathcal{G}(t)=\mathcal{G}_1(t)+\mathcal{G}_2(t), \label{Green}
\end{equation}
with 
\begin{equation}\label{G0}
\begin{array}{l}
\mathcal{G}_1(t)\psi:\displaystyle =  \frac{1}{2\pi i}\displaystyle\int\limits_{-i\infty }^{i\infty
} e^{-K(p) t} e^{px}\hat{u}_0(p)dp,\ \ 
\mathcal{G}_2(t)\psi:= \displaystyle\frac{1}{2 \pi }\int _0^\infty e^{-px} \mathcal K(t)\psi (p) dp ,
\end{array} 
\end{equation}
where 
\begin{equation}\label{4.3}
\begin{array}{l}
\displaystyle \tilde{\mathcal{K}}(t)\psi(p):=\frac{1}{2}e^{ip^2 t}\mathcal E^{-}(\psi)(p,i)+\frac{1}{2}e^{-ip^2t}\mathcal E^{-}(\psi)(p,-i) +\textbf{VP}  \int\limits_{-i\infty  }^{i\infty } e^{sp^2t} \frac{1}{1+s^2}\mathcal E^-(\psi)(p,s)ds,
\end{array}
\end{equation}
and
 \begin{equation}\label{E}
\mathcal E^{-}(\psi)(p,s):=\hspace{-0.5cm}\lim \limits_{\begin{array}{l}
\scriptstyle{s_1\to \varphi(s)}\\
\scriptstyle{\text{Re }s_1>0}
\end{array}}\frac{1}{2\pi i} \int_{-i\infty}^{i\infty}\widehat{\psi}(pw) \frac{e^{-\widetilde{\Gamma}(w,s)}-e^{-\widetilde{\Gamma}(0,s)} }{w-s_1}\Omega(w,s)
dw -e^{-\widetilde{\Gamma}(0,s)}\widehat{\psi}(\varphi(s)p)\Omega(\varphi(s),s),
 \end{equation}
 with 
\begin{equation}\label{f30}
\begin{array}{l}
\Omega(w,s)=\frac{1+\varphi(s)}{1+k(s)}e^{\widetilde{\Gamma}(-1,s)}\left( \frac{(1-\tilde{A}(s))k^2(s)}{1+k(s)\tilde{A}(s)} -\frac{w-k(s)}{w+ 1}\right) \\
 \tilde{A}(s)=\displaystyle\frac{s}{\pi}\int \limits_{\tilde{C}}\frac{c(q)}{(K(q)+s)(q^2+s)}dq+\frac{\varphi(s)-k(s)}{\varphi^2(s)}, \ c(q)=(1-i)\text{sign Im}(q),
\end{array}
\end{equation}
and the boundary operator is given by 
\begin{equation}\label{Boundary function}
\begin{array}{l}
\displaystyle  \mathcal{B}(t)h =\frac{1}{2\pi i}\int \limits_{0}^\infty e^{-px}p  (\widetilde{\mathcal B}(t)h)(p)dp=\int_0^t H(x,t-\tau)h(\tau)d\tau, 
\end{array}
\end{equation}
where 
\begin{equation}\label{Psi1}
\begin{array}{l}
\widetilde{\mathcal B}(t)h(p):=\displaystyle\left(e^{ip^2 t}\Psi_{B}(i)\hat{h}(ip^2)+e^{-ip^2 t}\Psi_{B}(-i)\hat{h}(-ip^2) +\frac{1}{ \pi i}
\textbf{ VP}\int\limits_{-i\infty }^{i\infty,
} e^{sp^2 t} \frac{\Psi_B(s)}{1+s^2}\hat{h}(sp^2)ds\right) \\
H(t)=\displaystyle\frac{1}{2\pi i}\int \limits_{0}^\infty e^{-px}p 
\displaystyle \left(e^{ip^2 t}\Psi_{B}(i)+e^{-ip^2 t}\Psi_{B}(-i) +\textbf{VP}\frac{1}{\pi}\int\limits_{-i\infty }^{i\infty
} e^{sp^2 t} \frac{\Psi_B(s)}{1+s^2}ds\right)dp\\
\displaystyle \displaystyle \Psi_{B}(s)=s\frac{1+k(s)}{1+k(s)\tilde{A}(s)}e^{\tilde{\Gamma}(-1,s)-\tilde{\Gamma}(0,s)}.
\end{array}
\end{equation}
Finally, the sectionally analytic function $\tilde{\Gamma}$ is given by 
\begin{equation}\label{gammatilde1}
\begin{array}{l}
\displaystyle\tilde{\Gamma}(w,\xi)=-\frac{1}{2\pi i}\int \limits _{\tilde{C}}\ln (q-w)d\ln\left[ \frac{K(q)+\xi}{\tilde{K}(q)+\xi}\right]dq.
\end{array}
\end{equation}
$$  \widetilde{C}=\left\{q=re^{i\theta}: r\in (0,\infty) \textnormal{ and } \theta=\pm \frac{3\pi}{8}\right\}.$$
Let $K(p):=-p|p|$, $\widetilde{K}(p):=-p^2$ and  let $\varphi(\xi)$, $k(\xi)$ be  the unique solutions of the equations $K(p)+\xi=0, \ \widetilde{K}(p)+\xi=0$ for Re $\xi>0$, with Re $\varphi(\xi)>0,$ Re $k(\xi)>0.$

In order to prove Lemma \ref{propo lineal} we reduce the linear problem \eqref{lineal} to the corresponding Riemann problem. This Riemann problem  has some additional necessary conditions for solvability due to the growth of the of the non-analytic non-homogeneous symbol $K(p)=-p|p|$.  Therefore, we will show below that exactly one
boundary value is  necessary and sufficient in the problem \eqref{nolineal} for its solvability and uniqueness. We call this procedure as the analytic
continuation method. Therefore, via the Duhamel principle, the 
IBV problem \eqref{nolineal} can be rewritten as the integral equation
\begin{equation*}
u=\mathcal M(u)=\mathcal G(t)u_0+\mathcal{B}(t)h
+\int\limits_{0}^{t}e^{-\tau K(p)}\mathcal G(t-\tau )u\partial_{x}u(\tau)d\tau.
\end{equation*}

For $1\leq r\leq s$, $\mu \in [0,1]$ such that $\frac{1}{s}\leq \frac{1}{r}+\mu $ and $n=0,1$, $\epsilon\in(0,\frac{1}{2})$ we prove the following estimates
\begin{equation*}
\begin{array}{l}
     \Vert\mathcal G(t)\psi\Vert_{\mathbf H^1}+
     \|\partial^n_x\mathcal{G}(t)\psi\|_{\mathbf L^{s}}\lesssim \Vert\psi\Vert_{\mathbf Z^{1+\epsilon,2}} +t^{-\frac{1}{2}(n+\frac{1}{r} +\mu-\frac{1}{s})}\|\psi\|_{\mathbf L^{r,\mu}},\\
     \\
     \Vert\mathcal B(t)h\Vert_{\mathbf H^1}+ \|\mathcal{B}(t)h\|_{\mathbf L^{2,\epsilon}(\mathbb R^+)} \lesssim \|h\|_{\mathbf Z^{1,1}}.
\end{array}
\end{equation*}
Applying these estimates we prove
that $\mathcal{M}$ is a contraction mapping on a ball $\mathbf{C}((0,\infty);\mathbf H^1(\mathbb R^+)).$

\section{Linear problem}\label{Lproblem}
We consider the linearised version of the problem, that  is the initial boundary value problem in \eqref{lineal}. We prove the following Lemma.

\begin{lemma}\label{propo lineal}
Suppose that the initial and boundary data $u_0$, $h$ belongs to $ \mathbf L^1(\mathbb R^+)$. Then, there exists a unique solution $u(x,t)$ of the initial-boundary value problem \eqref{lineal}, which has the following integral representation
$u(x,t)=\mathcal{G}(t)\psi+\mathcal{B}(t)h$, operators $\mathcal{G}(t)$ and $\mathcal{B}(t)$
was given by \eqref{Green}.
\end{lemma}

In order to prove Lemma \ref{propo lineal} we  recall  some basic results related with the analytic continuation method.

\begin{lemma}
 Let $\phi$ be a complex function, which obeys the H\"older condition for all finite $q,$ and tends to a definite limit  $\phi_{\infty}$ as $|q|\to \infty$, such that for large $q,$ the following inequality holds $|\phi(q)-\phi_\infty|\leq C|q|^{-\mu}$, $\mu>0$. Then 
\begin{equation}
\mathbb P{\phi} (z)=\frac{1}{2\pi i}\int\limits_{i\mathbb{R}}\frac{\phi (q)}{q-z}dq,\ \label{operator proy}
\textnormal{ Re }z\neq 0,
\end{equation}%
 is an analytic  function  in the left and right semiplanes. Here and below these functions will be denoted $\mathbb P^+\phi(z)$ and $\mathbb P^-\phi(z)$, respectively. These functions have the limiting values $\mathbb P^+\phi(p)$ and $\mathbb P^-\phi(p)$ at all points of imaginary axis Re $p = 0$, on approaching the contour from the left and from the right, respectively. These limiting values are expressed by Sokhotzki- Plemelj formula:
 \begin{equation}\label{2.3a}
 \begin{array}{c}
 \displaystyle \mathbb P^+\phi(p)=\lim \limits_{z\to p, \textnormal{Re }z<0}\mathbb P\phi(z)=\frac{1}{2\pi i} PV\int\limits_{-i\infty}^{i\infty}\frac{\phi(q)}{q-p}dq +\frac{1}{2}\phi(p),\\
 \displaystyle \mathbb P^-\phi(p)=\lim \limits_{z\to p, \textnormal{Re }z<0}\mathbb P\phi(z)=\frac{1}{2\pi i} PV\int\limits_{-i\infty}^{i\infty}\frac{\phi(q)}{q-p}dq -\frac{1}{2}\phi(p).
 \end{array}
 \end{equation}
\end{lemma}
Subtracting and adding the formula \eqref{2.3a} we obtain the following two equivalent
 formulas
 \begin{equation}\label{2.4a}
 \begin{array}{c}
 \displaystyle \mathbb P^+\phi(p)-\mathbb P^-\phi(p)=\phi(p),\\
  \displaystyle \mathbb P^+\phi(p)+\mathbb P^-\phi(p)=\frac{1}{\pi i} \textnormal{PV}\int\limits_{-i\infty}^{i\infty}\frac{\phi(q)}{q-p}dq, 
 \end{array}
 \end{equation}
 which will be frequently used hereafter.
 
 \begin{lemma}\label{SP lemma}
 An arbitrary function $\phi$ given on the contour Re $p=0$, satisfying the H\"older
condition, can be uniquely represented in the form
\begin{equation}
\phi(p)=U^+(p)-U^-(p),
\end{equation}
where  $U^{\pm}(p)$ are the boundary values of the analytic functions  $U^{\pm}z$ and the condition $U^{\pm}_{\infty}=0$ holds.
These functions are determined by formula
$U(z)=\mathbb P{\phi} (z)$.
 \end{lemma}
 
 \begin{lemma}[Index Zero]\label{IZero}
An arbitrary function $\phi$  given on the contour Re $ p = 0$, satisfying the H\"older
condition, and having zero index,
\begin{equation}
\textnormal{ind }\phi(t)=\frac{1}{2\pi i}\int\limits_{-i\infty}^{i\infty} d\ln \phi(p)=0,
\end{equation}
is uniquely representable as the ratio of the functions $X^+(p)$ and $X^-(p)$, constituting the boundary values of functions, $X^+(z)$ and $X^-(z)$, both of them analytic functions in the left and right complex semiplane, and non-vanishing  on these domains. These functions are determined  by the formula
\begin{equation}
X^{\pm}(z)=e^{\Gamma^{\pm}(z)},  \ \ \Gamma(z)=\mathbb P(\ln \varphi)(z).
\end{equation}
 \end{lemma}
The proof of these Lemmas can be found in \cite{Gakov}.

\begin{proof}[Proof of Lemma \ref{propo lineal}]
To derive an integral representation for the solution of the problem \eqref{lineal}, we adopt the analytic continuation
method proposed in  \cite{kaikina2015capillary}. We suppose that there exists a solution $u(x, t)$ for the problem \eqref{lineal}, such
that $u(x, t) = 0$, for all $x < 0.$
Note that the Laplace transform of $\mathcal H u_{xx}$ has the form 
$$\mathcal{L}(\mathcal H u_{xx})=\mathbb{P}^{-}\left\{K(p)\left(\hat u(p,t)-p^{-1}u(0,t)-p^{-2}u_x(0,t) \right) \right\},$$
where the operator $\mathbb P^{-}$ was defined by \eqref{operator proy} and $K(p)=-|p|p.$ Therefore, by applying the
Laplace transform with respect to both, the space and time variables in \eqref{lineal}, we obtain 
\begin{equation}\label{sl}
\widehat{\widehat{u}}\left( p,\xi \right) =\frac{1}{K\left( p\right) +\xi }%
\left[ \widehat{u}_{0}\left( p\right) +\frac{K\left( p\right) }{p}\widehat{u}%
\left( 0,\xi \right) +\frac{K\left( p\right) }{p^{2}}\widehat{u}_{x}\left( 0,\xi \right) +\widehat{\Psi }%
\left( p,\xi \right) \right] ,
\end{equation}
for some complex function $\Psi \left( p,\xi \right) =O(p^{-\delta }),\
\delta >0$, such that $\mathbb{P}^{-}\left\{ \Psi \left( p,\xi \right)
\right\} =0$.
Here, $\widehat{u}\left( 0,\xi \right) $ and $\widehat{u}%
_{x}\left( 0,\xi \right) $ are the Laplace transforms of the boundary data $u\left(
0,t\right) ,$ and $u_{x}\left( 0,t\right) $ respectively. Note that in \eqref{sl} we
have three unknown function $\widehat{u}\left( 0,\xi \right),$ $\widehat{u}%
_{x}\left( 0,\xi \right) $ and $\Psi \left( p,\xi \right),$ moreover the function that appears in the right part of this equality is not analytic when  $\textnormal{Re }p>a>0$. To solve these problems we need to introduce 
the “analyticity switching” functions $Y^\pm$.

Let us denote $ \widetilde{K}(p):=-p^{2}$. Note that, for Re $p > 0$ and Re $\xi > 0,$ the equality $\widetilde{K}(p) + \xi = 0$ has only
one root $k(\xi),$ such that $\textnormal{Re }k(\xi)>0$. We make a cut along  the negative real axis. We define 
$$\omega^+(p,\xi)=\left(\frac{p}{p-k(\xi)} \right)^{\frac{1}{2}}, \ \ \\ \ \omega^-(p,\xi)=\left(\frac{p}{p+k(\xi)} \right)^{\frac{1}{2}}.$$  
Since for Re $\xi>0,$ one has
$$\textnormal{ind}\frac{K(q)+\xi}{\tilde{K}(q)+\xi}\frac{\omega^-(q,\xi)}{\omega^+(q,\xi)}=\frac{1}{2\pi i}\int \limits_{ -i\infty}^{ i\infty}d\ln \left\{
\frac{K(q)+\xi}{\tilde{K}(q)+\xi}\frac{\omega^-(q,\xi)}{\omega^+(q,\xi)}\right\}=0,$$
 via Index Zero Lemma in the form of Lemma \ref{IZero}, we  have
\begin{equation}
\frac{K(p)+\xi }{\tilde{K}(p)+\xi }=\frac{Y^{+}(p,\xi )}{Y^{-}(p,\xi )},  \ \ p\in i \mathbb R,	
\label{fraccion}
\end{equation}
where%
\begin{equation}
Y^{\pm }=e^{\Gamma ^{\pm }(p,\xi )}w^{\pm }, \ \ \Gamma (z,\xi )=\mathbb {P}\ln\left\{ \frac{K(\cdot)+\xi }{\tilde{K}(\cdot)+\xi }\frac{\omega^-}{\omega^+}\right\}.
\label{sw}
\end{equation}
Via Lemma \ref{SP lemma} we have
\begin{equation}\label{l2}
\begin{array}{c}
U^{+}(p,\xi)-U^{-}(p,\xi)=\displaystyle \frac{\widehat{\psi}_0(p)}{Y^+(p,\xi)}, \ \  I^{+}(p,\xi)-I^{-}(p,\xi)=\frac{K(p)p^{-2}}{Y^+(p,\xi)},\\
J^{+}(p,\xi)-J^{-}(p,\xi)=\displaystyle\frac{K(p)p^{-1}}{Y^+(p,\xi)},
\end{array}
\end{equation}
where 
\begin{equation}\label{4.7}
\begin{array}{c}
U(z,\xi)=\displaystyle \mathbb{P}\left(\frac{\widehat{\psi}_0(p)}{Y^+(p,\xi)} \right), \ \    I(z,\xi)=\mathbb{P}\left( \frac{K(p)p^{-2}}{Y^+(p,\xi)}\right),\ \
J(z,\xi)=\displaystyle\mathbb{P}\left( \frac{K(p)p^{-1}}{Y^+(p,\xi)}\right),
\end{array}
\end{equation}
Applying \eqref{fraccion}, \eqref{l2} into \eqref{sl} we obtain for  $p\in i\mathbb{R}$ that
\begin{equation}
 \begin{array}{l}
\displaystyle \widehat{\widehat u}(p,\xi)=
\displaystyle\frac{Y^+(p,\xi)}{K(p)+\xi}\left[U^+-U^- +\frac{\widehat{\Phi}^+}{Y^+} +\left( J^{+}-J^{-}\right)\widehat{u}%
\left( 0,\xi \right) +\left( I^{+}-I^{-}\right)\widehat{u}_{x}\left( 0,\xi \right)\right].
 \end{array}
\end{equation}
By the analyticity of the function $\widehat{\widehat u}$ in the right half-plane, we must  put the following conditions
$$\frac{\widehat{\Phi}}{Y^+}+U^+ + J^{+}\widehat{u}(0,\xi)+I^{+}\widehat{u}_{x}(0,\xi)=0,$$ and
\begin{equation}\label{l3}
 U^-(k(\xi),\xi) +J^{-}(k(\xi),\xi)\widehat{u}\left( 0,\xi \right)+I^{-}(k(\xi),\xi)\widehat{u}_{x}\left( 0,\xi \right)=0,
\end{equation}
where $k(\xi)$ is the only one root of the equation $\widetilde{K}(p)+\xi=0$ in the right half complex plane. Because  $Y^+(p,\xi)$
is analytic for  $\textnormal{Re }p<0,$ we have $\mathbb P^{-}\{\Phi\}=0$, moreover from \eqref{l3} we observe that 
we need to put in the IBV problem, only one boundary data, and the other unknown boundary condition will be completely determined by this
equality.  Thus, if we consider the Dirichlet boundary condition, $u(0,t)=h(t),$ the other unknown boundary data
$\hat{u}_x(0,\xi)$ satisfies that
$$\hat{u}_x(0,\xi)=\frac{1}{I^{-}(k(\xi),\xi)}\left( U^-(k(\xi),\xi) +J^{-}(k(\xi),\xi)\widehat{h}\left(\xi \right) \right),$$ where $\hat{h}(\xi)$ is the Laplace transform of $h(t)$.
Finally, we obtain for  the solution of \eqref{lineal} that,
\begin{equation}\label{fin1}
\begin{array}{rl}
\hat{\hat{u}}(p,\xi )=&\displaystyle \frac{Y^{+}(p,\xi)}{K(p)+\xi }\left( \frac{I^{-}(p,\xi )}{I^{-}(k(\xi),\xi )}U^{-}(k(\xi ),\xi )-U^{-}(p,\xi) \right)\\
\\
&\displaystyle +\frac{Y^{+}(p,\xi)}{K(p)+\xi }\widehat{h}(\xi )\left( \frac{I^{-}(p,\xi )}{%
I^{-}(k(\xi ),\xi )}J^{-}(k(\xi ),\xi )-J^{-}(p,\xi )\right) . 
\end{array}
\end{equation}
 Applying $ \mathcal{L}^{-1}_x \mathcal L^{-1}_t$ to (\ref{fin1}), we get  
$u(x,t)=\mathcal{G}(t)\psi+\mathcal{B}(t)h,$
where%
\begin{equation}
\begin{array}{rl}\label{r.1.3}
\mathcal{G}(t)\psi&\displaystyle=\left(\frac{1}{2\pi i}\right)^{2}\int\limits_{-i\infty +\varepsilon }^{i\infty
+\varepsilon } e^{\xi t} \int\limits_{-i\infty }^{i\infty }e^{px}\frac{Y^{+}(p,\xi )}{K(p)+\xi }\\
&\displaystyle \hspace{5cm}\times\left( \frac{I^{-}(p,\xi )}{I^{-}(k(\xi ),\xi )}
U^{-}(k(\xi ),\xi,y )-U^{-}(p,\xi,y )\right) dp \ d\xi ,
\end{array}
\end{equation}
 and 
 \begin{equation}\label{r.1.4}
\mathcal{B}(t)h= \left(\frac{1}{2\pi i}\right)^{2}\int\limits_{-i\infty +\varepsilon }^{i\infty
+\varepsilon }e^{\xi t} \hat{h}(\xi) \int\limits_{-i\infty }^{i\infty }e^{px}\frac{Y^{+}(p,\xi )}{K(p)+\xi }%
\left(\frac{I^{-}(p,\xi )}{I^{-}(k(\xi ),\xi )}J^{-}(k(\xi
),\xi )- J^{-}(p,\xi )\right) dp \  d\xi. 
 \end{equation}
 Now, our goal is to find a  convenient expression for these operators. It follows from the Sokhotzki-Plemelj formula that $U^-(p,\xi)= U^+(p,\xi)-\frac{\widehat{\psi}_0(p)}{Y^+(p,\xi)}$, and also, via Fubini Theorem. By computing the residue at $\xi= -K(p),$  we obtain
\begin{equation*}
\begin{array}{rl}
\displaystyle \left(\frac{1}{2\pi i}\right)^{2} \int\limits_{-i\infty +\varepsilon }^{i\infty
+\varepsilon } e^{\xi t} \int\limits_{-i\infty }^{i\infty }e^{px}\frac{\widehat{\psi}(p)}{K(p)+\xi }dp \ d\xi& \displaystyle= \frac{1}{2\pi i} \int\limits_{-i\infty }^{i\infty
} e^{-K(p) t} e^{px}\hat{\psi}(p)dp.
\end{array}
\end{equation*}
Thus, by using these facts in \eqref{r.1.3} we have
\begin{equation}
\mathcal{G}(t)\psi = \mathcal{G}_1(t)\psi + \mathcal{G}(t)u_2, \label{r.1.5}
\end{equation}
where 
\begin{equation}\label{r.1.6.0}
\mathcal{G}_1(t)\psi:\displaystyle =  \frac{1}{2\pi i}\displaystyle\int\limits_{-i\infty }^{i\infty
} e^{-K(p) t} e^{px}\widehat{\psi}(p)dp,\\
\end{equation}
\begin{equation}\label{r.1.6}
\begin{array}{l}
\mathcal{G}_2(t)\psi:= \displaystyle \left(\frac{1}{2\pi i}\right)^{2}\int\limits_{-i\infty +\varepsilon }^{i\infty
+\varepsilon } e^{\xi t} \int\limits_{-i\infty }^{i\infty }e^{px}\frac{Y^{+}(p,\xi )}{K(p)+\xi }\\
\hspace{1cm}\displaystyle \hspace{5cm}\times\left( \frac{I^{-}(p,\xi )}{I^{-}(k(\xi ),\xi )}
U^{-}(k(\xi ),\xi,y )-U^{+}(p,\xi,y )\right) dp \ d\xi. 
\end{array}
\end{equation}
Now, by using some analytic properties we will rewrite this operator.  We consider an analytic extension of
$K(p)=-p|p|=-i\text{sign}(\text{Im }p)p^2,$  in the form
\begin{equation}\label{extention }
K(p)=\left\{ 
\begin{array}{cl}
-ip^2 ,&  \textnormal{Im }p>0,\\
\\
ip^2, & \textnormal{Im }p<0.
\end{array}\right.
\end{equation}
Using Cauchy Theorem  we get 
\begin{equation}\label{g1}
I^{-}(p,\xi)=-\frac{K(p)+\xi}{p^2}\frac{1}{Y^+(p,\xi)}+\frac{\xi}{p^2}\frac{1}{Y^+(0,\xi)}[1+p\tilde{A}(\xi)],
\end{equation}
where $\tilde{A}(\xi)=Y^+(0,\xi)\partial_w\frac{1}{Y^+(w,\xi)}|_{w=0}$. Because $K(k(\xi ))=-\xi,$ from the last identity we obtain
\begin{equation}\label{g2}
I^{-}(k(\xi),\xi)=\frac{\xi}{k^2(\xi)}\frac{1}{Y^+(0,\xi)}[1+k(\xi)\tilde{A}(\xi)].
\end{equation}
Therefore, via \eqref{g1} and \eqref{g2} we obtain
\begin{equation}\label{4.21}
\frac{I^{-}(p,\xi )}{I^{-}(k(\xi ),\xi )}=\frac{k^2(\xi)}{p^2}\frac{Y^+(0,\xi)}{\xi[1+k(\xi)\tilde{A}(\xi)]}\left[ -\frac{K(p)+\xi}{Y^+(p,\xi)}+\frac{\xi}{Y^+(0,\xi )}[1+p\tilde{A}(\xi)]\right].
\end{equation}
By substituting \eqref{4.21} into \eqref{r.1.6},  we obtain
\begin{equation} \label{SumaG}
\begin{array}{l}
\mathcal{G}_2(t)\psi=\displaystyle -\left( \frac{1}{2\pi i}\right)^2\int\limits_{-i\infty +\varepsilon }^{i\infty
+\varepsilon }d\xi e^{\xi t} \int\limits_{-i\infty }^{i\infty }e^{px}\frac{1}{K(p)+\xi }%
\left(Y^{+}(p,\xi )U^{+}(p,\xi )\right.\\
\hspace{1cm}\left. \displaystyle+\frac{k^2(\xi)}{ \xi}\frac{U^{-}(k(\xi ),\xi )}{(1+k(\xi)\tilde{A}(\xi))}\left[\frac{-K(p)Y^+(0,\xi)+\xi\left( (1+p\tilde{A}(\xi))Y^+(p,\xi)-Y^+(0,\xi) \right)}{p^2}\right] \right) dp.
\end{array}
\end{equation}
By observing that, for $\xi$ with Re  $\xi >0$, the function $\displaystyle p\mapsto \frac{1}{K(p)+\xi}$ is analytic in the left half-plane, except when $\{\textnormal{Re }p=0\},$  for an analytic function $\phi(p),$  the Cauchy Theorem 
implies
\begin{equation} \label{p1}
\frac{1}{2\pi i}\int \limits_{i\mathbb R}e^{px}\frac{\phi(p)}{K(p)+\xi}dp=-\frac{1}{\pi}\int_0^{\infty}e^{-px}\frac{p^2}{p^4+\xi^2}\phi(-p)dp,
\end{equation}
\begin{equation}\label{p2}
\frac{1}{2\pi i}\int \limits_{i\mathbb R}e^{px}\frac{K(p)}{K(p)+\xi}\frac{1}{p^2}dp=\frac{\xi}{\pi}\int\limits_0^{\infty}
e^{-px}\frac{1}{p^4+\xi^2}dp.
\end{equation}
Via \eqref{p1}, \eqref{p2} and Fubini theorem, we  have
\begin{equation}
\begin{array}{l}
\displaystyle \mathcal{G}_2(t)\psi=\frac{1}{2\pi^2 i}\int \limits_{0}^\infty e^{-px} \int\limits_{-i\infty +\varepsilon }^{i\infty
+\varepsilon } e^{\xi t} \frac{Y^+(-p,\xi)}{p^4+\xi^2}\\
\displaystyle \hspace{4cm}\times\left[ \frac{1-p\tilde{A}(\xi)}{1+k(\xi)\tilde{A}(\xi)} k^2(\xi)U^{-}(k(\xi ),\xi )-p^2U^+(-p,\xi)\right]d\xi\ dp.
\end{array}
\end{equation}
Now, let us simplify the function 
$$
\begin{array}{l}
\displaystyle G(p,t):=\int\limits_{-i\infty +\varepsilon }^{i\infty
+\varepsilon } e^{\xi t} \frac{Y^+(-p,\xi)}{p^4+\xi^2}
\left[ \frac{1-p\tilde{A}(\xi)}{1+k(\xi)\tilde{A}(\xi)} k^2(\xi)U^{-}(k(\xi ),\xi )-p^2U^+(-p,\xi)\right]d\xi\\
\hspace{1.4cm} \displaystyle = \frac{1}{2\pi i}\int\limits_{-i\infty +\varepsilon }^{i\infty
+\varepsilon } e^{\xi t} \frac{Y^+(-p,\xi)}{p^4+\xi^2} 

\int_{-i\infty}^{i\infty}\frac{\widehat{\psi}(q)}{Y^+(q,\xi)} 
\left( \frac{1-p\tilde{A}(\xi)}{1+k(\xi)\tilde{A}(\xi)} \frac{k^2(\xi)}{q-k(\xi)}-\frac{p^2}{q+p}\right) dq \ d\xi.
\end{array}
$$
For $p>0$, via Lemmas \ref{Lemma Gamma} and \ref{Lemma A}, we know $$\frac{Y^+(-p,sp^2)}{Y^+(q,sp^2)}=\frac{1+\varphi(s)}{1+k(s)}\frac{q-k(s)}{q-\varphi(s)}e^{\widetilde{\Gamma}(-1,s)-\widetilde{\Gamma}(q,s)},$$ and  $$p\widetilde{A}(sp^2)=\widetilde{A}(s),$$ and that $k(sp^2)=pk(s).$ Therefore, the change of variables $\xi=p^2 s,$ and $q=pw,$ imply
\begin{equation}\label{4.28 b}
\begin{array}{l}
\displaystyle G(p,t)= \displaystyle \frac{1}{2\pi i} \int\limits_{-i\infty +\varepsilon }^{i\infty
+\varepsilon } e^{sp^2t} \frac{1}{1+s^2}\frac{1+\varphi(s)}{1+k(s)}e^{\widetilde{\Gamma}(-1,s)}\\
\displaystyle \hspace{2cm}\times\int_{-i\infty}^{i\infty}\widehat{\psi}(pw) \frac{e^{-\widetilde{\Gamma}(w,s)} }{w-\varphi(s)}
\left( \frac{(1-\tilde{A}(s))k^2(s)}{1+k(s)\tilde{A}(s)} -\frac{w-k(s)}{w+ 1}\right) dw \ ds.
\end{array}
\end{equation}
Since $\widehat{\psi}(pw) $ is an analytic function in the right-half complex plane,  by Cauchy Theorem the following
equality
$$
\begin{array}{l}
\displaystyle \frac{1}{2\pi i}\int_{-i\infty}^{i\infty}\widehat{\psi}(pw) \frac{e^{-\widetilde{\Gamma}(0,s)} }{w-\varphi(s)}
\left( \frac{(1-\tilde{A}(s))k^2(s)}{1+k(s)\tilde{A}(s)} -\frac{w-k(s)}{w+ 1}\right) dw\\
\displaystyle \hspace{2cm}=-e^{-\widetilde{\Gamma}(0,s)}\widehat{\psi}(\varphi(s)p)\left( \frac{(1-\tilde{A}(s))k^2(s)}{1+k(s)\tilde{A}(s)} -\frac{\varphi(s)-k(s)}{\varphi(s)+ 1} \right) ,
\end{array}
 $$  holds  for Re $s>0$. As a consequence of this analysis, we can rewrite the function $G(p,t)$ in \eqref{4.28 b} in the form 
\begin{equation}
\begin{array}{rl}
\displaystyle G(p,t)=   \int\limits_{-i\infty +\varepsilon }^{i\infty+\varepsilon } e^{sp^2t} \frac{1}{1+s^2}\mathcal E^-(\psi)(p,s)ds,
\end{array}
\end{equation} 
 where 
 \begin{equation}
\begin{array}{rl}
 \mathcal E^{-}(\psi)(p,s)   &:=\displaystyle \lim \limits_{\begin{array}{l}
\scriptstyle{s_1\to \varphi(s)}\\
\scriptstyle{\text{Re }s_1>0}
\end{array}}\frac{1}{2\pi i} \int_{-i\infty}^{i\infty}\widehat{\psi}(pw) \frac{e^{-\widetilde{\Gamma}(w,s)}}{w-s_1}\Omega(w,s)
dw \\
     & \displaystyle  =\lim \limits_{\begin{array}{l}
\scriptstyle{s_1\to \varphi(s)}\\
\scriptstyle{\text{Re }s_1>0}
\end{array}}\frac{1}{2\pi i} \int_{-i\infty}^{i\infty}\widehat{\psi}(pw) \frac{e^{-\widetilde{\Gamma}(w,s)}-e^{-\widetilde{\Gamma}(0,s)} }{w-s_1}\Omega(w,s)
dw -e^{-\widetilde{\Gamma}(0,s)}\widehat{\psi}(\varphi(s)p)\Omega(\varphi(s),s),
\end{array}
 \end{equation}
 with 
 $$\Omega(w,s)=\frac{1+\varphi(s)}{1+k(s)}e^{\widetilde{\Gamma}(-1,s)}\left( \frac{(1-\tilde{A}(s))k^2(s)}{1+k(s)\tilde{A}(s)} -\frac{w-k(s)}{w+ 1}\right). $$

Now we apply the Sokhotskii-Plemelj formula to get
\begin{equation}
\begin{array}{rl}
G(p,t)\hspace{-0.3cm}&=\displaystyle 2\pi i\left(\frac{1}{2}e^{ip^2 t}\mathcal E^{-}(\psi)(p,i)+\frac{1}{2}e^{-ip^2t}\mathcal E^{-}(\psi)(p,-i) +\textbf{VP}  \int\limits_{-i\infty +\varepsilon }^{i\infty+\varepsilon }  \frac{e^{sp^2t}}{1+s^2}\mathcal E^-(\psi)(p,s)ds\right) \\
&=2\pi i [\mathcal K(t)\psi] (p).
\end{array}
\end{equation}
 Consequently,
 \begin{equation}\label{G2}
 \mathcal{G}_2(t)\psi=\frac{1}{2 \pi }\int _0^\infty e^{-px} [\mathcal K(t)\psi] (p) dp .
 \end{equation}
 Finaly, from \eqref{r.1.5},  \eqref{r.1.6.0} and  \eqref{G2} we arrive to the representation \eqref{G0} of $\mathcal{G}(t)u_0$, as we have claimed.

Now we simplify $\mathcal{B}(t)$ given by \eqref{r.1.4}. We rewrite   $\mathcal{B}(t)$ as follows
\begin{equation}
\mathcal{B}(t)h=\mathcal{B}_1(t)h+\mathcal{B}_2(t)h,
\end{equation} where 
\begin{equation} \label{4.34}
\begin{array}{l}
\mathcal{B}_1(t)h= \displaystyle\left(\frac{1}{2\pi i}\right)^{2}\int\limits_{-i\infty +\varepsilon }^{i\infty
+\varepsilon }e^{\xi t} \hat{h}(\xi) \int\limits_{-i\infty }^{i\infty }e^{px}\frac{Y^{+}(p,\xi )}{K(p)+\xi }%
\left(\frac{I^{-}(p,\xi )}{I^{-}(k(\xi ),\xi )}J^{-}(k(\xi
),\xi )\right) dp \  d\xi, \\
\mathcal{B}_2(t)h= \displaystyle\left(\frac{1}{2\pi i}\right)^{2}\int\limits_{-i\infty +\varepsilon }^{i\infty
+\varepsilon }e^{\xi t} \hat{h}(\xi) \int\limits_{-i\infty }^{i\infty }e^{px}\frac{Y^{+}(p,\xi )}{K(p)+\xi } J^{-}(p,\xi )dp \  d\xi.
\end{array}
\end{equation}
Using the Cauchy Theorem,  a direct calculation
gives
$$J^{-}(p,\xi)=\frac{K(p)+\xi}{p}\frac{1}{Y^+(p,\xi)}-\frac{\xi}{p}\frac{1}{Y^+(0,\xi)}.$$
Observing that $K(k(\xi))=-\xi,$ and by using the last identity, we obtain
$$J^{-}(k(\xi),\xi)=-\frac{\xi}{k(\xi)}\frac{1}{Y^+(0,\xi)}.$$
Substituting the previous formula into \eqref{4.34}, we obtain 
\begin{equation} \label{4.36}
\begin{array}{l}
\mathcal{B}_1(t)h= -\displaystyle\left(\frac{1}{2\pi i}\right)^{2}\int\limits_{-i\infty +\varepsilon }^{i\infty
+\varepsilon }e^{\xi t} \hat{h}(\xi) \frac{k(\xi)}{1+k(\xi)\tilde{A}(\xi)}\frac{1}{Y^+(0,\xi)} \\
\hspace{1.6cm}\displaystyle\times\int\limits_{-i\infty }^{i\infty }e^{px}\frac{1}{K(p)+\xi }
\left( \frac{-K(p)Y^+(p,\xi)+\xi \left( (1+p\tilde{A}(\xi))Y^+(p,\xi)-Y^+(0,\xi)\right)}{p^2}\right) dp \  d\xi, \\
\mathcal{B}_2(t)h= \displaystyle\left(\frac{1}{2\pi i}\right)^{2}\int\limits_{-i\infty +\varepsilon }^{i\infty
+\varepsilon }e^{\xi t} \hat{h}(\xi) \int\limits_{-i\infty }^{i\infty }e^{px}\frac{1}{K(p)+\xi }\left( \frac{K(p)}{p}+\frac{\xi}{p}\left( 1- \frac{Y^+(p,\xi)}{Y^+(0,\xi)}\right)\right)dp \  d\xi.
\end{array}
\end{equation}
By following a similar procedure to the one, followed in the analysis of $\mathcal{G}(t),$ applied to the previous representation, we obtain
\begin{equation*}
\begin{array}{l}
\mathcal B(t)h
\displaystyle =\frac{1}{2\pi^2 i}\int \limits_{0}^\infty e^{-px}\\
\hspace{1cm}\times\displaystyle \left(\frac{ e^{ip^2 t}\Omega_B(p,ip^2)\hat{h}(ip^2)+e^{-ip^2 t}\Omega(p,-ip^2)\hat{h}(-ip^2)}{p^2} +\textbf{VP}\int\limits_{-i\infty }^{i\infty
} e^{\xi t} \frac{\Omega_B(p,\xi)}{p^4+\xi^2}\hat{h}(\xi)d\xi\ dp \right),
\end{array}
\end{equation*}
where
\begin{equation*} 
\Omega_B(p,\xi)=\xi\frac{Y^{+}(p,\xi)}{Y^{+}(0,\xi)}\frac{p+k(\xi)}{1+k(\xi)\widetilde{A}(\xi)}.
\end{equation*}
Via Lemmas \ref{Lemma Gamma} and \ref{Lemma A}, we have $\Omega_B(p,sp^2)=p^3\Psi_B(s)$, where $\Psi_B$ is given by 
\eqref{Psi1}. 
Thus, the change of variables  $\xi:=sp^2$ implies
\begin{equation}
\begin{array}{l}
\mathcal B(t)h=\displaystyle\frac{1}{2\pi i}\int \limits_{0}^\infty e^{-px}p \\
\hspace{1.5cm}\times\displaystyle \left(e^{ip^2 t}\Psi_{B}(i)\hat{h}(ip^2)+e^{-ip^2 t}\Psi_{B}(-i)\hat{h}(-ip^2) +\textbf{VP}\frac{1}{\pi}\int\limits_{-i\infty }^{i\infty
} e^{sp^2 t} \frac{\Psi_B(s)}{1+s^2}\hat{h}(sp^2)ds\ dp\right),
\end{array}
\end{equation} 
and moreover 
\begin{equation}
\begin{array}{l}
\displaystyle \mathcal B(t)h= \int_0^t H(t-\tau)h(\tau)d\tau, 
\end{array}
\end{equation} 
where 
$$H(t)=\displaystyle\frac{1}{2\pi i}\int \limits_{0}^\infty e^{-px}p \\
\displaystyle \left(e^{ip^2 t}\Psi_{B}(i)+e^{-ip^2 t}\Psi_{B}(-i) +\textbf{VP}\frac{1}{\pi}\int\limits_{-i\infty }^{i\infty
} e^{sp^2 t} \frac{\Psi_B(s)}{1+s^2}ds\right) dp,$$
which implies the integral representation \eqref{Boundary function}. This analysis  completes the proof of Lemma \ref{propo lineal}.
 
\end{proof}

\section{Preliminaries}\label{LemmasTecnicos}
In this section we present some essential lemmas for our further analysis.
Firstly, we prove the main properties of the operators $\mathcal{G}(t)$ and $\mathcal{B}(t)$ defined in \eqref{Green} and \eqref{Boundary function} respectively. 

\begin{lemma}\label{Lemma 5.1}
For $\psi\in \mathbf Z^{1+\epsilon,2}:=\mathbf{H}^{1+\epsilon}(\mathbf R^+)\cap \mathbf L^{2}(\mathbb R^+,|x|^{2}dx)$,  with $\epsilon\in (0,\frac{1}{2})$, $n=0,1,$ $\psi(0)=0,$  the estimate
$\|\partial_x^n\mathcal{G}(t)\psi \|_{\mathbf L^2(\mathbf R^+)}\lesssim \langle t \rangle^{-\frac{1}{2}(n+\frac{1}{2})}\|\psi\|_{\mathbf{Z}^{1+\epsilon, 2}(\mathbb R^+)}$
holds.
\end{lemma}

\begin{proof}
Let us recall that $\mathcal G(t)=\mathcal G_1(t)+ \mathcal G_2(t)$, where $\mathcal G_1$ and $\mathcal G_2$ are defined in \eqref{G0}.

 Let $\psi^*=\mathbf{1}_{(0,\infty)}\psi$ be an extension of $\psi$ from $\mathbb{R}^+$ to $\mathbb{R}$, since $\psi(0)=0.$ By following   
 \cite[Lemma 2.1]{Erdogan}, we have that
  $\|\psi^*\|_{\mathbf Z^{1+\epsilon}(\mathbb R)}\leq \|\psi\|_{\mathbf Z(\mathbb R^+)}$, where $\mathbf Z^{s,r}(\Omega):=\mathbf{H}^r(\Omega)\cap \mathbf L^{2}(\Omega, |x|^{2r}dx)$. Then $\mathcal{G}_{\mathbb R}(t)\psi^*$ is defined for all $x\in \mathbb R$ and $\mathcal{G}_{1}(t)\psi=\left.\mathcal{G}_{\mathbb R}(t)\psi^*\right|_{=0}$. By using  Theorem 1 in \cite{FG},
we have that $\mathcal{G}_{\mathbb R}(t)\psi^* \in \mathbf{C}([0,\infty); \mathbf{Z}^{1,2})$ and also that
$\| \mathcal{G}_{\mathbb R}(t)\psi^*\|_{\mathbf{H}^{1 }}\lesssim \|\psi^*\|_{\mathbf{Z}^{1,2}(\mathbb R)}\lesssim \|\psi\|_{\mathbf Z^{1,2}(\mathbb R^+)} \lesssim \|\psi\|_{\mathbf Z^{1+\epsilon,2}(\mathbb R^+)}.$ This analysis leads us  to conclude that $\mathcal{G}_{1}(t)\psi \in\mathbf{C}([0,\infty); \mathbf{H}^{1}). $

Now, we estimate the $\mathbf{H}^1$-norm of  $\mathcal{G}_2(t)$. First, we note that for Re$(w)=0$, and Im$(w)>0,$ by a slightly abuse of notation, $$\widehat{\psi}(pw)=\mathcal{F}\{\mathbf{1}_{(0,\infty)}\psi\}(-p|w|),$$
where $\mathcal{F}$ stands for the Fourier transform, and $c(w):=-\text{sign Im}( w).$ Since $\|\partial_x^n e^{-px}\|_{\mathbf {L}^2(\mathbb R^+)}\lesssim p^{n-\frac{1}{2}} $  for $p>0$,  we have 
\begin{equation}
\begin{array}{rl}
\displaystyle\left\| \partial_x^n \int_0^\infty e^{-px}\widehat {\psi}(pw)dp\right\|_{\mathbf {L}^2_x(\mathbb R^+)}&\displaystyle
\lesssim  \int_0^\infty p^{n-\frac{1}{2}}|\mathcal{F}\{\mathbf{1}_{(0,\infty)}\psi\}(-p|w|)dp\\
&\displaystyle\lesssim |w|^{-(n+\frac{1}{2})} \int_{-\infty}^\infty |q|^{n-\frac{1}{2}}|\mathcal{F}\{\mathbf{1}_{(0,\infty)}\psi\}(q)|dq\\
&\displaystyle\lesssim	|w|^{-(n+\frac{1}{2})}  \left( \|\mathbf{1}_{(0,\infty)}\psi\|_{\mathbf L^{2,1}(\mathbb R)}+\|\mathbf{1}_{(0,\infty)}\psi\|_{\mathbf H^{1+\epsilon}(\mathbb R)}\right).
\end{array}
\end{equation}
It is clear that  $ \|\mathbf{1}_{(0,\infty)}\psi\|_{\mathbf L^{2,1}(\mathbb{R})}\leq  \|\psi\|_{\mathbf L^{2,1}(\mathbf R^+)}.$ Also, note that $\|\mathbf{1}_{(0,\infty)}\psi\|_{\mathbf H^{1+\epsilon}(\mathbb R)} \lesssim \|\psi\|_{\mathbf H^{1+\epsilon}(\mathbb R^+)},$ in view of Lemma 2.1 in \cite{Erdogan}. Therefore, 
\begin{equation}\label{5.2}
\left\| \partial_x^n \int_0^\infty e^{-px}\widehat {\psi}(pw)dp\right\|_{\mathbf {L}^2_x(\mathbb R^+)} \lesssim 	|w|^{-(n+\frac{1}{2})}  \|\psi\|_{\mathbf Z^{1+\epsilon,2}(\mathbb R^+)}.
\end{equation}
On the other hand, the change of variables $w=q\sqrt{|s|}$ applied to the function $\mathcal{E}^-$, defined in \eqref{E}, implies the estimate
\begin{equation}\label{E1}
\displaystyle \mathcal E^{-}(\psi)(p,s):=\mathcal E^{-}_{1}(\psi)(p,s)+ \mathcal E_{2}(\psi)(p,s),
\end{equation}
where
$$
\begin{array}{l}
\mathcal E^{-}_{1}(\psi)(p,s)=\hspace{-0.5cm}\displaystyle \lim \limits_{\begin{array}{l}
\scriptstyle{s_1\to \varphi(s)}\\
\scriptstyle{\text{Re }s_1>0}
\end{array}}\frac{1}{2\pi i} \int_{-i\infty}^{i\infty}\widehat{\psi}(pq\sqrt{|s|}) \frac{e^{-\widetilde{\Gamma}(q\sqrt{|s|}),s)}-e^{-\widetilde{\Gamma}(0,s)} }{q	-s_1|s|^{-\frac{1}{2}}}\Omega(q\sqrt{|s|}),s)
dw \\
 \mathcal E_{2}(\psi)(p,s)=-e^{-\widetilde{\Gamma}(0,s)}\widehat{\psi}(\varphi(s)p)\Omega(\varphi(s),s)\\
 \Omega(w,s)=\frac{1+\varphi(s)}{1+k(s)}e^{\widetilde{\Gamma}(-1,s)}\left( \frac{(1-\tilde{A}(s))k^2(s)}{1+k(s)\tilde{A}(s)} -\frac{w-k(s)}{w+ 1}\right).
\end{array}
$$
Since $\tilde A(s)=O(\frac{1}{\sqrt{|s|}})$ and $k(s)=O(\sqrt{|s|})=\varphi(s),$ we have $ \Omega(q\sqrt{|s|}),s)=O(\{s\}^{\frac{1}{2}}\langle s \rangle)$.  From which we conclude that 
\begin{equation}\label{E_2}
\begin{array}{l}
\displaystyle \left\| \partial_x^n \int_0^\infty e^{-px} \mathcal E^{-}_{1}(\psi)(p,s)dp\right\|_{\mathbf {L}^2_x(\mathbb R^+)}\\
 \hspace{2cm} \lesssim \displaystyle \|\psi\|_{\mathbf{Z}^{1+\epsilon, 2}(\mathbb R^+)}\frac{\{s\}^{\frac{1}{2}}\langle s \rangle }{|s|^{\frac{1}{2}(n+\frac{1}{2})}}\displaystyle \lim \limits_{\begin{array}{l}
\scriptstyle{s_1\to \varphi(\pm i)}\\
\scriptstyle{\text{Re }s_1>0}
\end{array}} \frac{1}{2\pi i} \int_{-i\infty}^{i\infty}\frac{\{q\} }{|q|^{n+\frac{1}{2}}|q-s_1|}
dw\\
 \hspace{2cm}  \lesssim \displaystyle \|\psi\|_{\mathbf{Z}^{1+\epsilon, 2}(\mathbb R^+)}\{s\}^{-\frac{1}{4}}\langle s\rangle^{\frac{3}{4}}.
\end{array}
\end{equation}
A similar argument to the one used above, allows us to show that 
\begin{equation}
\left\| \partial_x^n \int_0^\infty e^{-px} \mathcal E_{2}(\psi)(p,s)dp\right\|_{\mathbf {L}^2_x(\mathbb R^+)} \lesssim \displaystyle \|\psi\|_{\mathbf{Z}^{1+\epsilon, 2}(\mathbb R^+)}\{s\}^{-\frac{1}{4}}\langle s\rangle^{\frac{3}{4}}.
\end{equation}
From  \eqref{E1}-\eqref{E_2} we conclude the estimate $\left\| \partial_x^n \int_0^\infty e^{-px} \mathcal E^-(\psi)(p,s)dp\right\|_{\mathbf {L}^2_x(\mathbb R^+)} \lesssim \displaystyle \|\psi\|_{\mathbf{Z}^{1+\epsilon, 2}(\mathbb R^+)}\{s\}^{-\frac{1}{4}}\langle s\rangle^{\frac{3}{4}}.$ By combining this inequality with the definition of $\mathcal{G}_2(t)$, given by \eqref{G0}, we have
$\|\partial_x^n\mathcal{G}_2(t)\psi \|_{\mathbf L^2(\mathbf R^+)}\lesssim \|\psi\|_{\mathbf{Z}^{1+\epsilon, 2}(\mathbb R^+)} $, thus we have 
\begin{equation}\label{p 1}
\|\partial_x^n\mathcal{G}(t)\psi \|_{\mathbf L^2(\mathbf R^+)}\lesssim \|\psi\|_{\mathbf{Z}^{1+\epsilon, 2}(\mathbb R^+)}.
\end{equation}

On the other hand, Hayashi and Kaikina in \cite{hayashi2010benjamin}  proved
\begin{equation} \label{p 2}
\|\partial^n_x\mathcal{G}(t)\psi\|_{\mathbf L^{2}}\lesssim t^{-\frac{1}{2}(n+\frac{1}{2})}\|\psi\|_{\mathbf L^{1}},
\end{equation}
since $\|\psi\|_{\mathbf L^{1}}\lesssim \|\psi\|_{\mathbf Z^{1,1}}$, from \eqref{p 1} and \eqref{p 2} we can conclue 
\begin{equation}
\|\partial_x^n\mathcal{G}(t)\psi \|_{\mathbf L^2(\mathbf R^+)}\lesssim \langle t \rangle^{-\frac{1}{2}(n+\frac{1}{2})}\|\psi\|_{\mathbf{Z}^{1+\epsilon, 2}(\mathbb R^+)}.
\end{equation}

\end{proof}

Our attention now is turned to the boundary operator $\mathcal{B}(t)$, given by \eqref{Boundary function}. In the next Lemma we present an estimation of the $\mathbf{H}^1$-norm of $\mathcal{B}(t)h$.

\begin{lemma}\label{Lemma 5.3}
For $h\in\mathbf Z^{1,1}:= \mathbf H^1(\mathbb R^+)\cap \mathbf L^{2,1}(\mathbf R^+)  $, with $h(0)=0$ we have
$$\Vert\mathcal B(t)h\Vert_{\mathbf H^1}\lesssim \langle t \rangle^{-(\frac{3}{4}+\frac{n}{2})}  \Vert h \Vert_{\mathbf Z^{1,1}}.
$$
\end{lemma}

\begin{proof}
For any Re$(s)=0$, we observe that $\widehat{h}(p^2s)=\mathcal{F}\{\mathbf{1}_{(0,\infty)}h\}( c(s)p^2|s|)$, where $\mathcal F$ is the clasical Fourier transform and $c(s)=\text{sign Im}(s)$. Since $\|\partial_x^n e^{-px}\|_{\mathbf L^2_x(\mathbb R^+)}\lesssim p^{n-\frac{1}{2}}$ for $p>0,$ we deduce the estimates
\begin{equation}\label{5.14}
\begin{array}{rl}
\displaystyle\left\| \partial_x^n \int_0^\infty e^{-px}p \widehat {h}(p^2s)dp\right\|_{\mathbf {L}^2_x(\mathbb R^+)}&\displaystyle
\lesssim  \int_0^\infty p^{n+\frac{1}{2}}|\mathcal{F}\{\mathbf{1}_{(0,\infty)}h\}(c(s)p^2|s|)|dp\\
&\displaystyle\lesssim |s|^{-(\frac{3}{4}+\frac{n}{2})} \int_{-\infty}^\infty |w|^{\frac{1}{2}(n-\frac{1}{2})}|\mathcal{F}\{\mathbf{1}_{(0,\infty)}h\}(w)|dw\\
&\displaystyle\lesssim	|s|^{-(\frac{3}{4}+\frac{n}{2})}\|\mathbf{1}_{(0,\infty)}h\|_{\mathbf H^{1}(\mathbb R)}\lesssim |s|^{-(\frac{3}{4}+\frac{n}{2})}\|h\|_{\mathbf H^{1}(\mathbb R^+)}.
\end{array}
\end{equation}

Since $\Psi_B(s)=O(\{s\}\langle  s\rangle^{\frac{3}{2}})$, from \eqref{Boundary function} and \eqref{5.14} we get
\begin{equation}
\|\partial_x^n  \mathcal{B}(t)h\|_{\mathbf {L}^2_x(\mathbb R^+)} \lesssim \|h\|_{\mathbf{H}^1(\mathbf R^+)}\left(1+\textbf{VP} \int_{-i\infty}^{i\infty}\frac{1}{|s|^{\frac{3}{4}+\frac{n}{2}}}\frac{\{s\}\langle s \rangle^{\frac{3}{2}}}{|1+s^2|} ds \right)\lesssim \|h\|_{\mathbf{H}^1(\mathbf R^+)}.
\end{equation}
On the other hand, for $t>1$, let us remember the equivalent representation for the operator $\mathcal{B}(t)$
\begin{equation}
\mathcal{B}(t)h=\int_0^tH(t-\tau)h(\tau)d\tau.
\end{equation}
 Via the  Cauchy theorem and the change of variables  $p=t^{\frac{1}{2}}w$,  we rewrite the function $H$, given by \eqref{Psi1}, as $H=H_1+H_2+H_3$
where 
$$
\begin{array}{l}
\displaystyle H_1(x,t)=t^{-1}\int_{\widetilde{C}_+} e^{-pxt^{-\frac{1}{2}}}p e^{ip^2}dp\\
\displaystyle H_2(x,t)=t^{-1}\int_{\widetilde{C}_-} e^{-pxt^{-\frac{1}{2}}}p e^{-ip^2} dp\\
\displaystyle H_3(x,t)=t^{-1}\int_{0}^\infty e^{-pxt^{-\frac{1}{2}}}p \int_{\widetilde C} e^{isp^2 }\frac{\Psi_B(is)}{1-s^2}ds
\end{array}
$$
where $\widetilde C:=\{q\in \mathbb C: \text{arg}(q)=\delta, \pi-\delta\}$ and  $\widetilde{C}_{\pm}:=\{q\in \mathbb C: \text{arg}(q)=\pm \delta\}$.  It is important to observe for $p\in \widetilde{C}_{\pm}$ we have Re$(pxt^{-\frac{1}{2}})>0$, Re$(\pm ip^2)<0$
where $\delta>0$ is sufficiently small. Moreover for $p>0$ and $s\in \widetilde C$ we note Re$(isp^2 )<0$.  Since $\|\partial_x^n e^{-px}\|_{\mathbf L^2_x(\mathbb R^+)}\lesssim |p|^{n-\frac{1}{2}}$ for Re$p>0$ we get
 \begin{equation}
\begin{array}{rl}
\displaystyle\left\| \partial_x^n H_1\right\|_{\mathbf {L}^2_x(\mathbb R^+)}& \displaystyle \lesssim t^{-\frac{3}{4}-\frac{n}{2}} \int_{\widetilde C_{+}} |p|^{n+\frac{1}{2}}e^{-|p|^2}dp
 \lesssim t^{-\frac{3}{4}-\frac{n}{2}}
\end{array}
\end{equation}
In an analogous way, we get $\left\| \partial_x^n H_2\right\|_{\mathbf {L}^2_x(\mathbb R^+)} \lesssim t^{-\frac{3}{4}-\frac{n}{2}}$.
To control the $\mathbf{L^2}-$norm of $\partial_x^n H_3$ we observe
$$\int_{0}^\infty \int_{\widetilde C} e^{-|s|p^2 }p^{n+\frac{1}{2}}\frac{\{s\}\langle s \rangle ^{\frac{3}{2}}}{|1-s^2|}ds \ dp \lesssim 1,$$ therefore $\left\| \partial_x^n H_3\right\|_{\mathbf {L}^2_x(\mathbb R^+)} \lesssim t^{-\frac{3}{4}-\frac{n}{2}}$. From previous inequalities we infer
\begin{equation}
\left\| \partial_x^n H(\cdot, t)\right\|_{\mathbf {L}^2_x(\mathbb R^+)} \lesssim t^{-\frac{3}{4}-\frac{n}{2}}.\label{q 2}
\end{equation}
We  also note that
\begin{equation}
\label{q 1}  \mathcal{B}(t)h= \int_0^tH(x,t-\tau)h(\tau)d\tau=\int_0^{\frac{t}{2}} H(x, t-\tau)h(\tau)d\tau + \int_{\frac{t}{2}}^t H(x, \tau)h(t-\tau)d\tau, 
\end{equation}
hence, \eqref{q 1} and \eqref{q 2} implies 
$$ \left\| \partial_x^n \mathcal{B}(t)h \right\|_{\mathbf {L}^2_x(\mathbb R^+)}  \lesssim \int_0^{\frac{t}{2}} \frac{1}{(t-\tau)^{\frac{3}{4}+\frac{n}{2}}}h(\tau)d\tau + \int_{\frac{t}{2}}^t \frac{1}{\tau^{\frac{3}{4}+\frac{n}{2}}}h(t-\tau)d\tau\lesssim t^{-\frac{3}{4}-\frac{n}{2}}\|h\|_{\mathbf{L}^1(\mathbb R^+)},$$ 
since $\|h\|_{\mathbf{L}^1(\mathbb R^+)}\lesssim \|h\|_{\mathbf{Z}^{1,1}(\mathbb R^+)}$, we have  $\left\| \partial_x^n \mathcal{B}(t)h \right\|_{\mathbf {L}^2_x(\mathbb R^+)}  \lesssim t^{-\frac{3}{4}-\frac{n}{2}}\|h\|_{\mathbf{Z}^{1,1}(\mathbb R^+)}$. The proof is complete.
\end{proof}

The following Lemma is used to obtain an estimate in the proof of Proposition \ref{Prop 2}.

\begin{lemma}\label{Lemma 5.4}
For $h\in \mathbf Z^{1,1}= \mathbf H^1(\mathbf R^+)\cap \mathbf{L}^{2,1}(\mathbf R^+) $  we have
\begin{equation} \label{5.37}
\|\mathcal{B}(t)h\|_{\mathbf L^{2,\epsilon}_x(\mathbb R^+)}\lesssim \|h\|_{\mathbf Z^{1,1}}.
\end{equation}
\end{lemma}

\begin{proof}
For $p>0,$ we observe: $\| e^{-px}\|_{\mathbf L^{2,\epsilon}_x(\mathbb R^+)}\lesssim \langle p \rangle^{\epsilon} p^{-\frac{1+2\epsilon}{2}}.$ So, we can estimate 
\begin{equation}\label{5.17}
\begin{array}{l}
\displaystyle \left\| \int_0^\infty e^{-px}p e^{\pm ip^2 t}\Psi_{B}(\pm i)\hat{h}(\pm ip^2) dp\right\|_{\mathbf {L}^{2,\epsilon}_x(\mathbb R^+)}\displaystyle \lesssim  \int_{0}^\infty \{ p \} ^{\frac{1}{2}-2\epsilon}\langle p \rangle ^{\frac{1}{2}-\epsilon}|\mathcal{F}\{\mathbf{1}_{(0,\infty)}h\}(p^2))|dp\\
 \hspace{2cm}\displaystyle \lesssim \int_{-\infty}^\infty \{ w \} ^{-(\frac{1}{2}+\epsilon)}\langle w \rangle ^{\frac{1}{4}(1+2\epsilon)}|\mathcal{F}\{\mathbf{1}_{(0,\infty)}h\}(w))|dp\\
   \hspace{2cm} \displaystyle \lesssim  \|\mathcal{F}\{\mathbf{1}_{(0,\infty)}h\}\|_{\mathbf L^\infty}\int_{|w|<1} w ^{-(\frac{1}{2}+\epsilon)}dw + \|w \mathcal{F}\{\mathbf{1}_{(0,\infty)}h\}\|_{\mathbf L^2}\int_{|w|\geq 1} |w|^{-(\frac{7}{2}+2\epsilon)} dw\\
      \hspace{2cm} \displaystyle \lesssim  \|h\|_{\mathbf L^{2,1}(\mathbb R^+)}+\|h\|_{\mathbf H^{1}(\mathbf R^+)}.
\end{array}
\end{equation}
Now, we focus our attention to the function 
$$\int_0^\infty e^{-px}pF(p)dp, \ \text{ where } F(p):=\textbf{VP}\int_{-i \infty}^{i\infty } e^{isp^2 t}\frac{\Psi_B(s)}{1+s^2ds}\widehat{h}(sp^2)ds.$$
Since $\Psi(s)=O(\{s\}\langle  s \rangle^{\frac{3}{2}})$, fix $a\in(0,1]$, the Cauchy–Schwartz inequality implies that
$$
\begin{array}{l}
\displaystyle  |F(p)|\lesssim \left[ \textbf{VP}\int_{-i\infty}^{i\infty} \frac{\{s\}^2\langle  s \rangle^{3}}{|s|^{2a}|1+s^2|^2}ds \right]^{\frac{1}{2}}\left[\int_{-i\infty}^{i\infty}|s|^{2a}|\widehat{h}(sp^2)|^2ds \right]^{\frac{1}{2}}\\
\displaystyle \hspace{1.2cm} \lesssim   \frac{1}{|p|^{2a+1}}\left[\int_{-\infty}^{\infty}|w|^{2a}|\mathcal{F}\{\mathbf{1}_{(0,\infty)}h\}(w))|^2ds \right]^{\frac{1}{2}}  \lesssim   \frac{1}{|p|^{2a+1}} \|h\|_{\mathbf H^1(\mathbb R^+)}.
\end{array}
$$ This analysis  implies the estimate $|F(p)|\lesssim \{p\}^{-(1+2a(\epsilon))}\langle p \rangle^{-3} \|h\|_{\mathbf H^1(\mathbb R^+)}, $ where $2a(\epsilon)\in(0, \frac{1}{2}-\epsilon)$.  Consequently,
\begin{equation}\label{5.18}
\begin{array}{rl}
\displaystyle \left\| \int_0^\infty e^{-px}p F(p)dp\right\|_{\mathbf {L}^{2,\epsilon}_x(\mathbb R^+)}
\hspace{-0.2cm}& \displaystyle \lesssim \int_0^\infty  \langle p \rangle^{\epsilon} p^{\frac{1-2\epsilon}{2}}|F(p)|dp\\
& \displaystyle  \lesssim \|h\|_{\mathbf H^1(\mathbb R^+)}  \int_0^\infty \langle p \rangle^{\epsilon} p^{\frac{1-2\epsilon}{2}} \{p\}^{-(1+2a(\epsilon))}\langle p \rangle^{-3} dp \lesssim \|h\|_{\mathbf H^1(\mathbb R^+)}  .
\end{array}
\end{equation}
From \eqref{5.17} and \eqref{5.18} the result follows.
\end{proof}

Finally, we will prove two lemmas that we have used in the estimation of $\mathbf{H}^1$-norm of 
$\mathcal{G}(t)$ and $\mathcal{B}(t)$.

Next, we consider the “analyticity switching” function  $Y^+(p,\xi)=e^{\Gamma^+(p,\xi)}\omega^+$,
where
\begin{equation}
\begin{array}{c}
\displaystyle \omega^+(p,\xi)=\left(\frac{1}{p-k(\xi)} \right)^{\frac{1}{2}}, \ \ \  \ \omega^-(p,\xi)=\left(\frac{1}{p+k(\xi)} \right)^{\frac{1}{2}}\\
\\
\displaystyle \Gamma (z,\xi )=\frac{1}{2\pi i}%
\int\limits_{-i\infty}^{i \infty }\frac{1}{q-z}\ln\left\{ \frac{K(q)+\xi }{\widetilde K(q)+\xi }\frac{\omega^-}{\omega^+}\right\}dq, \ K(q)=-q|q|, \  \  \widetilde{K}(q)=-q^2.
\end{array}
\end{equation}

\begin{lemma}\label{Lemma Gamma}
The following formula is valid for $w, \xi \in \mathbb{C}$, with Re $w<0$ and Re $\xi\geq0$, and $\text{arg}(\xi)\in(\frac{3\pi}{8},  \frac{15\pi}{8})$,
\begin{equation} \label{Function Y}
Y^+(w,\xi)=e^{\tilde{\Gamma}(w,\xi)}\frac{w-\varphi(\xi)}{w-k(\xi)},
\end{equation}
where $k(\xi)$, $\varphi(\xi)$ are the roots of the equations 
$K(q)+\xi=0$, $\tilde{K}(q)+\xi=0$ respectively, and 
\begin{equation}\label{gammatilde}
\tilde{\Gamma}(w,\xi)=-\frac{1}{2\pi i}\int \limits _{\tilde{C}}\ln (q-w)d\ln\left[ \frac{K(q)+\xi}{\tilde{K}(q)+\xi}\right]dq,
\end{equation}
with $\widetilde{C}=\left\{q=re^{i\theta}: r\in (0,\infty) \textnormal{ and } \theta=\pm \frac{\pi}{4}\right\}$.
Moreover, 
\begin{equation}\label{5.6a}
Y^+(qp,s p^2)=p^{\frac{3}{2}}e^{\tilde{\Gamma}(q,s)}\frac{q-\varphi(s)}{q-k(s)}.
\end{equation}
\end{lemma}

\begin{proof}
Using the analytical properties of the integrand function, after of integrating by parts, for Re $w<0$,  we rewrite $\Gamma^{+}$ as
\begin{equation}\label{5.62}
\begin{array}{rl}
\Gamma^+(w,\xi)=&\hspace{-0.3cm}\displaystyle-\frac{1}{2\pi i}\int \limits _{-i \infty}^{i \infty}\ln (q-w)d\ln\left[ \frac{K(q)+\xi}{\tilde{K}(q)+\xi}\frac{\omega^-}{\omega^+}\right]dq.
\end{array}
\end{equation}
By taking  the residue at the points $q=\varphi(\xi)$ and $q=k(\xi)$,the Cauchy Theorem implies 
$$\Gamma^+(w,\xi)= \displaystyle \tilde{\Gamma}(w,\xi)+\ln(\varphi(\xi)-w)-\ln(k(\xi)-w)-\ln(\omega^+).$$ The claim \eqref{Function Y}  follows.
In order to prove \eqref{5.6a}, let us note that after of the change of variables $q=pw$, we have
\begin{equation*}
\tilde{\Gamma}(qp,\pm sp^2)=\frac{1}{2\pi i}\int \limits _{\tilde{C}}\left[\ln(p)+\ln (q-w)\right]d\ln\left[ \frac{K(w)+ s}{\tilde{K}(w)+ s}\right]dw.
\end{equation*}
Since
$$\textnormal{ind }\left( \frac{K(w)+ s}{\tilde{K}(w)+ s}\right)=\frac{1}{2\pi i}\int \limits _{\tilde{C}}d\ln\left[ \frac{K(w)+s}{\tilde{K}(w)+ s}\right] =\left.\frac{1}{2\pi }\textnormal{Arg }\frac{K(w)+s}{\tilde{K}(w)+ s}\right|_{\tilde{C}}=-\frac{3}{2},$$
the analysis above  implies that $e^{ \tilde{\Gamma}(qp,sp^2)}=p^{\frac{3}{2}}e^{\tilde{\Gamma}(q,s)}.$
Since $k(sp^2)=k(s)p, \ \ \varphi(sp^2)=\varphi(s)p,$ via \eqref{Function Y} we conclude \eqref{5.6a}.
\end{proof}

\begin{lemma}\label{Lemma A}
The function $\tilde A(s)=Y^+(0,\xi)\left.\partial_w\frac{1}{Y^+(w,\xi)}\right|_{w=0}$ satisfies the identity
\begin{equation}\label{tilde A}
\tilde{A}(s)=\displaystyle\frac{s}{\pi}\int \limits_{\tilde{C}}\frac{c(q)}{(K(q)+s)(q^2+s)}dq+\frac{\varphi(s)-k(s)}{\varphi^2(s)}, \ c(q)=(1-i)\text{sign Im}(q).
\end{equation}
Moreover, $\tilde{A}(p^2s)=\frac{1}{p}\tilde{A}(s)$ for all $p>0.$
\end{lemma}
\begin{proof}
By a simple calculation, it follows that $\tilde A(s)=-\frac{k(s)}{\varphi(s)} \left.\partial_w\tilde{\Gamma}(w,s)\right|_{w=0}+\frac{k(s)-\varphi(s)}{\varphi(s)^2}$. This because of 
$$\left.\partial_w\tilde{\Gamma}(w,s)\right|_{w=0}:=\Gamma_1(s)=\frac{s}{ \pi i}\int\limits_{\tilde C} \frac{c(q)}{(K(q)+s)(-q^2	+s)}dq,$$
where $c(q)=(1-i)\text{sign Im}(q)$. Using these identities we conclude \eqref{tilde A}.  

In order to prove the second equality, we note that $c(qp)=c(q)$ for $p>0.$ Thus, the change of variables $q=pq_1,$ implies that
$\Gamma_1(sp^2)=\frac{1}{p}\Gamma_1(s).$ Since $k(sp^2)=k(s)p$ and $ \varphi(sp^2)=\varphi(s)p, $ we conclude that $p\tilde{A}(p^2s)=\tilde{A}(s)$. 
\end{proof}

\begin{lemma}
For Re$(w)\leq 0$ and Re$(\xi)>0$ with $\text{arg}(\xi)\in(\frac{3\pi}{8},  \frac{15\pi}{8})$, we have the following estimates:
\begin{itemize}
\item[i)] $|e^{\widetilde{\Gamma}(-1,\xi)}|\leq C.$

\item[ii)] $|e^{\widetilde{\Gamma}(w,\xi)}- e^{\widetilde{\Gamma}(0,\xi)}|\leq C\{w\}^{\frac{1}{2}+\gamma}$,  with $\gamma\in (0,\frac{1}{2}).$
\end{itemize}
\end{lemma}
\begin{proof}
Note that  $\ln(q+1)=O(\{q\}\langle q\rangle^{\frac{1}{2}})$, therefore 
$$|{\widetilde{\Gamma}(-1,\xi)}|\lesssim \frac{1}{2\pi }\int_{\tilde{C}}\{q\}\langle q\rangle^{\frac{1}{2}} \left| d\ln\left(  \frac{K(q)+\xi}{\tilde{K}(q)+\xi}\right)\right|dq.$$
Takin into account that $d\ln\left(  \frac{K(q)+\xi}{\tilde{K}(q)+\xi}\right)=O(\frac{|\xi|}{|q|^3})$,  the previous inequality implies the estimate:
$|{\widetilde{\Gamma}(-1,\xi)}|\lesssim |\xi|.$ But, we also can observe that  $d\ln\left(  \frac{K(q)+\xi}{\tilde{K}(q)+\xi}\right)=O(\frac{1}{|\xi|^{\frac{1}{2}}|q|^3})$, which implies the inequality, $|{\widetilde{\Gamma}(-1,\xi)}|\lesssim |\xi|^{-\frac{1}{2}}$. Therefore  $|{\widetilde{\Gamma}(-1,\xi)}|\lesssim \{\xi\}\langle \xi \rangle^{-\frac{1}{2}}.$ Thus we infer i). 

To obtain ii),  note that  $\ln(\frac{q+w}{q})=O(\left| \frac{w}{q}\right|^{\frac{1}{2}+\gamma})$ for $|w|<<1$  with $\gamma<1.$ Thus, 
 we get
$$
\begin{array}{rl}
|\tilde{\Gamma}(w,\xi)-\tilde{\Gamma}(0,\xi)|\hspace{-0.2cm}&\displaystyle\lesssim |w|^{\frac{1}{2}+\gamma}\int \limits_{\tilde C} \frac{1}{|q|^{\frac{1}{2}+\gamma}}\frac{|qs|}{|q^2+s|^2}dw
\lesssim |w|^{\frac{1}{2}+\gamma}.
\end{array}
$$
All this implies the estimate $|e^{\widetilde{\Gamma}(w,\xi)}- e^{\widetilde{\Gamma}(0,\xi)}| \lesssim |w|^{\frac{1}{2}+\gamma}.$ Moreover, in view of i) we  get $|e^{\widetilde{\Gamma}(w,\xi)}- e^{\widetilde{\Gamma}(0,\xi)}| \lesssim 1$ for $w>>1$, so we have ii).
\end{proof}
The next inequality  will be used to estimate the nonlinear term.
\begin{lemma} \label{desigualdad1}
 For $a,b\in \mathbb R$ such that $a+b>1$ we have
 \begin{equation}
  \int \limits_{\mathbb R}\frac{1}{\langle\tau-c_1 \rangle^a \langle \tau-c_2 \rangle^b }\leq C \frac{1}{\langle c_1-c_2 \rangle^{\delta}}
 \end{equation}
with $\delta= \min \{a,b,a+b-1\}$.
\end{lemma}

\section{Proof of Theorem \ref{Main result}} \label{TheProofofLili}
The proof of Theorem \ref{Main result} will be exposed in two propositions. In the first of them, Proposition \ref{Prop 1}, we prove there exist, under certain conditions of the initial and boundary conditions, a solution in the functional space $u\in \mathbf{C}([0,T]: \mathbf{H}^1(\mathbb{R^+}))$. On the other hand, in  Proposition \eqref{Prop 2}, we show that this solution belongs to $\mathbf{C}([0,T]: \mathbf{L}^{2,1}(\mathbb{R^+}))$.
\begin{proposition} \label{Prop 1}
For any $\psi\in \mathbf Z^{1+\epsilon,2}$ and $h\in  \mathbf Z^{1,1}$, with $\epsilon\in(0,\frac{1}{2})$, the IBVP \eqref{nolineal} has a unique global solution $u\in \mathbf{C}([0,T]: \mathbf{H}^1(\mathbb{R^+}))$.
\end{proposition}
\begin{proof}

By Proposition \ref{propo lineal} we rewrite the initial-boundary value problem \eqref{nolineal}
as the following integral equation 
\begin{equation}
u(t)=\mathcal{G}(t)u_0-\int_0^t\mathcal{G}(t-\tau)\mathcal{N}(u(\tau))d\tau+\mathcal{B}(t)h.
\end{equation}
For this, our goal is to prove that the transformation
\begin{equation}
\mathcal{M}v(t):=\mathcal{G}(t)u_0-\int_0^t\mathcal{G}(t-\tau)\mathcal{N}(v(\tau))d\tau+\mathcal{B}(t)h,
\end{equation}
is a contraction on  the ball $\mathbf{X}_\rho=\{v\in \mathbf X: \|v\|_{\mathbf X}\leq \rho(\|u_0\|+\|h\|)\},$ for $\rho>0,$ to be determined later, in the functional space 
$$\mathbf{X}:=\{v\in \mathbf C([0,\infty):\mathbf{H}^1(\mathbf R^+)): \|v\|_{\mathbf X}:=\sup_{t\geq 0} \left[ \langle t \rangle^{\frac{1}{4}}\|v(t)\|_{\mathbf{L}^2} +\langle t \rangle^{\frac{3}{4}}\|\partial_x v(t)\|_{\mathbf{L}^2} \right]<\infty\}.$$ 

We first prove that $\|\mathcal{M}v\|_{\mathbf X}\leq \rho$, where $\rho>0,$ is small enough. In view  of Lemma \ref{Lemma 5.1} we have $
\|\partial_x^n\mathcal{G}(t)u_0\|_{\mathbf{L}^2}\lesssim \langle t \rangle^{-\frac{1}{4}[2n+1]}\|u_0\|_{\mathbf{Z}}.
$
Therefore
\begin{equation}\label{6.3}
\|\mathcal{G}(t)u_0\|_{\mathbf{X}}\lesssim \|u_0\|_{\mathbf{Z}}.
\end{equation}
Similarly, Lemmas \ref{Lemma 5.3} and \ref{Lemma 5.4} imply 
\begin{equation}\label{6.4}
\|\mathcal{B}(t)h\|_{\mathbf X}\lesssim \|h\|_{\mathbf{Y}}.
\end{equation}
Also, note that if $v\in \mathbf{X}_\rho$, via  \ref{Lemma 5.1},  we have 
\begin{equation}
\begin{array}{rl}
\|\partial_x^n\mathcal{G}(t-\tau)\mathcal{N}(v(\tau))\|_{\mathbf{L}^2}& \lesssim
\langle t-\tau \rangle^{-\frac{1}{4}[2n+1]}\|\mathcal{N}(v(\tau))\|_{\mathbf{L}^1}\\
&\lesssim
\langle t-\tau \rangle^{-\frac{1}{4}[2n+1]} \|v(\tau)\|_{\mathbf{L}^2}\|\partial_x v(\tau)\|_{\mathbf{L}^2}\\
&\lesssim
\langle t-\tau \rangle^{-\frac{1}{4}[2n+1]}\langle  \tau \rangle^{-1} \|v\|^2_{\mathbf{X}}.
\end{array}
\end{equation}
As consequence, via Lemma \ref{desigualdad1}  we obtain:
\begin{equation}\label{6.6}
\left\| \partial_x^n \int_0^t\mathcal{G}(t-\tau)\mathcal{N}(v(\tau))d\tau \right\|_{\mathbf{L^2}} \lesssim
\|v\|^2_{\mathbf X}\int_0^\infty  \langle t-\tau \rangle^{-\frac{1}{4}[2n+1]}\langle  \tau \rangle^{-1} d\tau \lesssim \|v\|^2_{\mathbf X} \langle t\rangle^{-\frac{1}{4}[2n+1]}
\end{equation}
thus we conclude the uniform estimate,
\begin{equation}\label{6.10}
\left\| \int_0^t\mathcal{G}(t-\tau)\mathcal{N}(v(\tau))d\tau \right\|_{\mathbf{X}} \lesssim
\|v\|^2_{\mathbf X}. 
\end{equation}
Hence, in view of \eqref{6.3}, \eqref{6.4} and \eqref{6.10}, we deduce the estimate 
$$\|\mathcal{M}v\|\lesssim \|u_0\|_{\mathbf{Z}}+ \|h\|_{\mathbf{Y}} +\|v\|^2_{\mathbf X}  <\rho$$
for $\rho>0,$ small enough. Hence, the mapping $\mathcal M$ transforms the   ball $\mathbf X_\rho$ into itself. In the
same way,  we estimate the difference
$$\|\mathcal{M}v-\mathcal{M}w\|\leq\frac{1}{2}\|v-w\|,$$
which shows that $\mathcal{M}$ is a contraction mapping. Therefore we see that there exists a unique
solution $u\in\mathbf{C}([0,\infty): \mathbf H^1(\mathbb R^+))$.
\end{proof}
In the next proposition we prove that  suitable conditions on the initial  and boundary data, the solution to \eqref{nolineal} belongs to $\mathbf{C}([0,\infty: \mathbf Z)$. The  proof is based in the argument of G. Fonseca and G. Ponce in \cite{FG}. We start it, by defining the truncated weights 
$w_N(x)$ as
\begin{equation}
w_N(x)=\left\{ 
\begin{array}{rcl}
\langle x \rangle, & \text{if } |x|\leq N,\\
2N,& \text{if } |x|\geq  3N,
\end{array}
\right.
\end{equation}
$w_N(x)$ are smooth functions, and they are requested to be non-decreasing in $|x|$ with $w'_N (x)\leq 1 $ for all $x \geq 0$.
In \cite{FG}, it was proved that for any $\theta\in[-1,1],$ and for  any $N\in \mathbb{Z}^+$, $w^{\theta}_{N}(x)$ is a $\theta$-weight, i.e. it satisfies that
\begin{equation}
\sup_{Q \text{ interval }}\left( \frac{1}{|Q|}\int_Q w^{\theta}_N\right)\left(\frac{
1}{|Q|}\int_{Q}w^{-\theta}_N\right)=c(w, \theta)<\infty.
\end{equation}
Moreover, the Hilbert transform $ H$ on the whole line, which is the singular integral defined by
\begin{equation}
   {H}f(x):=\textnormal{PV}\int\limits_{-\infty}^\infty\frac{f(y)}{x-y}dy,\,\,f\in C^{\infty}_0(\mathbb{R}),
\end{equation}
extends to a bounded operator on ${L}^2(w^\theta_N(x)dx),$ for any $\mathbb{N},$ with the operator norm bounded by a constant depending on $\theta,$ but, independent of $N\in \mathbb N$ (see e.g. Proposition 1 in \cite{FG}). Due to the Riesz-Kolmogorov theorem, $H$ extends to a bounded operator on $\mathbf L^p(\mathbb{R}),$ for all $1<p<\infty,$ (see e.g. Duoandikoetxea \cite{Duoandikoetxea}).  Observe that the Hilbert transform $\mathcal{H}$ on the half line $\mathbf{R}^{+}$ can be obtained from the Hilbert transform $H$ on the whole line by the identity $\mathcal{H}f=(H\mathcal{R}f)|_{(0,\infty)},$ where $\mathcal{R}$ is the multiplication operator by the characteristic function $1_{(0,\infty)}$ of $\mathbb{R}^+,$  i.e. $\mathcal{R}f:=f\times 1_{(0,\infty)}.$  

\begin{proposition}\label{Prop 2}
For  $u_0\in \mathbf Z, h \in \mathbf Y$,    the solution $u(t,x)\in \mathbf{C}([0,\infty);\mathbf L^{2,\epsilon}(\mathbf{R}^+)).$ 
\end{proposition}
\begin{proof}
Let $u=v+z$ be such that $z$ satisfies 
\begin{equation} 
\begin{cases}
z_{t}+\mathcal{H}z_{xx}=0, \\
z(x,0)=0,\\
z(0,t)=h(t),\ t > 0, & \textnormal{ } \end{cases}
\end{equation}
and $w$ solves
\begin{equation} \label{nolineal 2}\begin{cases}
v_{t}+\mathcal{H}v_{xx}+u\partial_{x}u=0,\\
v(x,0)=\psi(x),\\
v(0,t)=0,\ & \textnormal{ } \end{cases}
\end{equation}
In Lemma \ref{Lemma 5.4} we had proved
\begin{equation}\label{boundforz(t)}
\sup \limits_{0\leq t\leq T}\|z(t)\|_{\mathbf L^{2,\epsilon}}\lesssim \|h\|_{\mathbf{Y}}.
\end{equation}
By multiplying the differential equation \eqref{nolineal 2} by $w^{2\theta}_Nv$ with $0<\theta\leq 1$ and integrating it on $[0,\infty)$ we have,
\begin{equation}
\frac{1}{2}\frac{d}{dt}\int\limits_0^\infty \left(w^{\theta}_N v \right)^2 dx+\int\limits_0^\infty w^{\theta}_N\mathcal{H}\partial_x^2 v  w^{\theta}_N v dx +\int\limits_0^\infty w^{2\theta}_N vu\partial_x u dx=0.
\end{equation}The previous identity implies that
\begin{align*}
  \frac{d}{dt} \Vert v w_{N}^\theta\Vert^2_{\mathbf{L}^{2}(\mathbf R^+)}\leq   |I'(t)|+|II'(t)|,\quad 
\end{align*}
and consequently 
\begin{align*}
   \Vert v w_{N}^\theta\Vert^2_{\mathbf{L}^{2}(\mathbf R^+)}\leq   \Vert xu_0\Vert^2_{\mathbf{L}^{2}(\mathbf R^+)} +|I|+|II|,\quad 
\end{align*}
with
\begin{align*}
I\equiv I(t):=\int_0^t\left( \int\limits_0^\infty \left( w^{\theta}_N\mathcal{H}\partial_x^2 v  w^{\theta}_N v \right)dx \right)dt ,\,\quad
 II\equiv II(t):=\int_0^t\left( \int\limits_0^\infty\left( w^{2\theta}_N vu\partial_x u \right)dx\right)dt.
\end{align*}
\begin{description}

\item \textbf{Estimation of $I$.} Let us use the Calder\'on commutator technique, for this we denote $\tilde{v}:=\mathbb{E}(v),$ where $\mathbb{E}$ is the  the extension operator by zero  to $\mathbb{R}^{-}_0:=(-\infty,0],$ (this means, $\tilde{v}(x)=v(x),$ for $x>0$ and when $x\leq 0,$ $\tilde{v}(x)=0$). Then we have
\begin{align*}
  & \int\limits_0^\infty w^{\theta}_N\mathcal{H}\partial_x^2 v  w^{\theta}_N v dx \\
  &=\int\limits_0^\infty w^{\theta}_N\mathcal{H}\partial_x^2 v  w^{\theta}_N v dx-\int\limits_0^\infty \mathcal{H}(w_{N}^\theta \partial_x^2 v) ( w^{\theta}_N v) dx +\int\limits_0^\infty \mathcal{H}(w_{N}^\theta \partial_x^2 v) ( w^{\theta}_N v)dx\\
   &=\int\limits_0^\infty [w^{\theta}_N,\mathcal{H}]\partial_x^2 v \times  w^{\theta}_N v dx+\int\limits_0^\infty \mathcal{H}(w_{N}^\theta \partial_x^2 v) ( w^{\theta}_N v)dx=A_{1}+III,
\end{align*} with $$  A_{1}:=\int\limits_0^\infty [w^{\theta}_N,\mathcal{H}]\partial_x^2 v \times  w^{\theta}_N v dx=\int\limits_{-\infty}^\infty [w^{\theta}_N,{H}](\partial_x^2  \tilde{v}) \times  w^{\theta}_N \tilde{v} dx,$$
where the derivative $\partial_x^2  \tilde{v}$ is understood in the sense of distributions. From \cite{Calderon,Dawson},  Calder\'on's commutator theorem allows us to estimate $[w^{\theta}_N,\mathcal{H}]$ as $$\Vert [w^{\theta}_N,H]\tilde{v}\Vert_{\mathbf L^{2}(\mathbb{R})}\lesssim \Vert \partial_{x}^2 w^{\theta}_N  \Vert_{\mathbf  L^\infty(\mathbb{R})}\Vert w^{\theta}_N \tilde{v}  \Vert_{\mathbf L^2(\mathbb{R})}, $$ 
and consequently,
\begin{align*}
    |A_1| &\lesssim \Vert [w^{\theta}_N,H]v\Vert_{\mathbf  L^{2}(\mathbb{R})}\Vert w^{\theta}_N \tilde{v}  \Vert_{\mathbf   L^2(\mathbb{R})}\lesssim  \Vert \partial_{x}^2 w^{\theta}_N  \Vert_{\mathbf  L^\infty(\mathbf R^+)}\Vert w^{\theta}_N \tilde{v}  \Vert_{\mathbf  L^2(\mathbb{R})}\Vert w^{\theta}_N {v}  \Vert_{\mathbf  L^2(\mathbf{R}^{+})}\\
    &=\Vert \partial_{x}^2 w^{\theta}_N  \Vert_{\mathbf  L^\infty(\mathbf R^+)}\Vert w^{\theta}_N {v}  \Vert^2_{\mathbf  L^2(\mathbf{R}^{+})}.
\end{align*} Now, let us estimate $III$ as follows. Observing that
\begin{align*}
    &\int\limits_0^\infty \mathcal{H}(w_{N}^\theta \partial_x^2 v) ( w^{\theta}_N v)dx &\\
    &=\int\limits_0^\infty \mathcal{H}(w_{N}^\theta \partial_x^2 v-\partial_{x}^2(w_{N}^\theta v)) ( w^{\theta}_N v)dx+\int_{0}^\infty \mathcal{H}\partial_{x}^2(w_{N}^\theta v)( w^{\theta}_N v)dx\\
    &=\int\limits_0^\infty \mathcal{H}(w_{N}^\theta \partial_x^2 v) ( w^{\theta}_N v)dx   -\int\limits_0^\infty  \mathcal{H}\partial_{x}^2(w_{N}^\theta v) ( w^{\theta}_N v)dx+\int_{0}^\infty \mathcal{H}\partial_{x}^2(w_{N}^\theta v)( w^{\theta}_N v)dx\\
     &=\int\limits_0^\infty \mathcal{H}(w_{N}^\theta \partial_x^2 v) ( w^{\theta}_N v)dx   -\int\limits_0^\infty  \mathcal{H}((\partial_{x}^2w_{N}^\theta)\times v) ( w^{\theta}_N v)dx\\
     &-2\int_{0}^{\infty}\mathcal{H}(\partial_x w_{N}^\theta \partial_x v) w_{N}^\theta  vdx-\int\limits_0^\infty \mathcal{H}(w_{N}^\theta \partial_x^2 v) ( w^{\theta}_N v)dx  +\int_{0}^\infty \mathcal{H}\partial_{x}^2(w_{N}^\theta v)( w^{\theta}_N v)dx\\
     &=  A_2 -\int\limits_0^\infty  \mathcal{H}((\partial_{x}^2w_{N}^\theta)\times v) ( w^{\theta}_N v)dx-2\int_{0}^{\infty}\mathcal{H}(\partial_x w_{N}^\theta \partial_x v)w_{N}^\theta  v dx =A_{2}+A_{3}+A_4 ,
\end{align*} where $A_2:=\int_{0}^\infty  \mathcal{H}\partial_{x}^2(w_{N}^\theta v)( w^{\theta}_N v)dx$ and $A_4:=-2\int_{0}^{\infty}\partial_x w_{N}^\theta \partial_x v dx.$ Let us study the terms $A_2,A_3$ and $A_4$ as follows.
\begin{itemize}
    \item[(i).] {\bf Estimating $A_2.$} Indeed,  since $H^{*}=-H,$ (see e.g. Duoandikoetxea \cite{Duoandikoetxea}), we obtain
\begin{align*}
 A_2:&=\int_{0}^\infty  \mathcal{H}\partial_{x}^2(w_{N}^\theta v)( w^{\theta}_N v)dx=\int\limits_{-\infty}^\infty  {H}\partial_{x}^2(w_{N}^\theta \tilde{v})( w^{\theta}_N \tilde{v})dx\\
 & =-\int\limits_{-\infty}^\infty  \partial_{x}^2(w_{N}^\theta \tilde{v})H( w^{\theta}_N \tilde{v})dx=
 -\int\limits_{0}^\infty  \partial_{x}^2(w_{N}^\theta v )\mathcal H( w^{\theta}_N v)dx,
 \end{align*}and consequently, since $v(t,0)=0$ via integration by parts we conclude that
 \begin{align}
 A_2=\frac{1}{2}v_{x}(t,0)\mathcal H(w_{N}^\theta v)(t,0). \label{6.18}
\end{align}
Note that $|w^{\theta}_N v(y)| \lesssim (1+y^\theta) |v(t,y)|.$ Therefore, 
\begin{equation}
\begin{array}{rl}
|\mathcal H(w_{N}^\theta v)(t,0)|&\hspace{-0.2cm} \displaystyle \lesssim \lim \limits_{\epsilon\to 0}\int \limits_\epsilon^\infty
\frac{(1+y^\theta) |v(t,y)|}{y}dy\\
&\hspace{-0.2cm} \displaystyle \lesssim \lim \limits_{\epsilon\to 0}\left( \int \limits_\epsilon^1
\frac{ |v(t,y)|}{y}dy + \int \limits_1^\infty
\frac{ |v(t,y)|}{y^{1-\theta}}dy \right)\\
&\hspace{-0.2cm} \displaystyle \lesssim \lim \limits_{\epsilon\to 0}\left(\|v(t)\|_{\mathbf{H}^1} \int \limits_\epsilon^1
\frac{ 1}{\sqrt{y}}dy +\|v(t)\|_{\mathbf{L}^2} \int \limits_1^\infty
\frac{ 1}{y^{2(1-\theta)}}dy \right)\\
&\hspace{-0.2cm} \displaystyle \lesssim \|v(t)\|_{\mathbf{H}^1} .
\end{array}
\end{equation}
On the other hand, since Duhamel's principle  gives us an integral representation of $v$, indeed
$$ v(t)=\mathcal{G}(t)u_0-\int_0^t\mathcal{G}(t-\tau)\mathcal{N}(u(\tau))d\tau,$$
by  combining this integral representation with  Lemma  \ref{Lemma 5.1}, and the fact that $\|N(u)\|_{\mathbf{L}^2}\lesssim \|u\|_{\mathbf{L}^\infty}\|\partial_x u\|_{\mathbf{L}^2}\lesssim \|u\|_{\mathbf H^1}\leq C$, where  $C$ is a constant  depending on both the initial and boundary data depend, we have the estimate
\begin{equation}
\begin{array}{rl}\label{6.20}
|v_{x}(t,0)|\lesssim \|v_x(t,x)\|_{\mathbf L^\infty} \lesssim \displaystyle C\left( t^{-\frac{3}{4}}+ \int_0^t (t-\tau)^{\frac{3}{4}}d\tau\right)\lesssim C\{ t\}^{-\frac{3}{4}}\langle t \rangle ^{\frac{1}{4}}.
\end{array}
\end{equation}
From \eqref{6.18}-\eqref{6.20} we have
\begin{equation}
|A_2|\lesssim C\{ t\}^{-\frac{3}{4}}\langle t \rangle ^{\frac{1}{4}}.
\end{equation}

\item[(ii).] {\bf Estimating $A_3.$} Because  $$A_3=-\int\limits_0^\infty  \mathcal{H}((\partial_{x}^2w_{N}^\theta)\times v) ( w^{\theta}_N v)dx=-\int\limits_{-\infty}^\infty  {H}((\partial_{x}^2w_{N}^\theta)\times \tilde{v}) ( w^{\theta}_N \tilde{v})dx,$$  from the $\mathbf L^2$-boundedness of the Hilbert transform we obtain,
\begin{align*}
    |A_3|\lesssim \int\limits_{-\infty}^\infty|(\partial_{x}^2w_{N}^\theta)\times v||w^{\theta}_N \tilde{v})|dx\lesssim \Vert \tilde{v}\Vert_{ \mathbf L^2(\mathbb{R})}\Vert w^{\theta}_N \tilde{v}\Vert_{\mathbf{L}^{2}(\mathbb{R})}\lesssim  \Vert w^{\theta}_N {v}\Vert_{\mathbf{L}^{2}(\mathbf R^+)}^2. 
\end{align*}
\item[(iii).] {\bf Estimating $A_4.$} Finally, using Ponce and Fonseca  \cite[Page 445]{FG}, we have the following estimate for the non-linear term $A_4.$ Indeed,
\begin{align*}
    |A_4|&\lesssim \int_{-\infty}^{\infty}|\mathcal{H}(\partial_x w_{N}^\theta \partial_x \tilde{v})w_{N}^\theta  \tilde{v}|dx \\
    &\lesssim \| \tilde{v}  \|^2_{\mathbf H^1(\mathbb{R})}+\|w_{N}^\theta  \tilde{v}\|^2_{\mathbf L^2(\mathbb{R})}= \| {v}  \|^2_{\textbf{H}^1(\mathbf R^+)}+\|w_{N}^\theta  {v}\|^2_{\textbf{L}^2(\mathbf R^+)}.
\end{align*} 
\end{itemize}
\item \textbf{Estimation of $II$.} Note, we can rewrite $II$ since
$$II=\int_0^\infty w_{N}^{2\theta} u^2\partial_x u dx -\int_0^\infty w_{N}^{2\theta} zu\partial_xu dx=II_1 +II_2$$
From, (3.2) in Ponce and Fonseca  \cite[Page 445]{FG}, we have
\begin{align}\label{6.23}
     |II_1|\leq  \int\limits_{-\infty}^\infty |w^{2\theta}_N \tilde{u}^2\partial_x \tilde{u}| dx\lesssim \Vert \langle x\rangle \tilde{u}\Vert_{L^{2}(\mathbb{R})}\|w_{N}^\theta  \tilde{u}\|_{L^2(\mathbb{R})}=\Vert \langle x\rangle u\Vert_{\textbf{L}^{2}(\mathbf R^+)}\|w_{N}^\theta  {u}\|_{\textbf{L}^2(\mathbf R^+)}.
\end{align}
Moreover via H\"older inequality we get 
\begin{align} \label{6.24}
     |II_2|\leq \|u\|_{\mathbf L^2(\mathbb R^+)}^2 \|w_N^{2\theta}z\|_{\mathbf L^2} \leq \|u\|_{\mathbf L^2(\mathbb R^+)}^2 \|z\|_{\mathbf L^{2,1}}.
\end{align}
From \eqref{6.23} and \eqref{6.24}, we deduce the estimate
$$|II|\lesssim \|u\|_{\mathbf L^2(\mathbb R^+)}^2 \|z\|_{\mathbf L^{2,1}}+ \Vert w^{\theta}_N {v}  \Vert^2_{\mathbf L^2(\mathbf{R}^{+})} +\Vert z \Vert^2_{\mathbf L^{2,1}(\mathbf{R}^{+})} .$$
\end{description}
Summarising, we have proved that
\begin{align*}
 \frac{d}{dt} \Vert v w_{N}^\theta\Vert^2_{\mathbf{L}^{2}(\mathbf R^+)}\lesssim C(t)+\Vert w^{\theta}_N {v}  \Vert^{2}_{L^2(\mathbf{R}^{+})}+\| {v}  \|^2_{\textbf{H}^1(\mathbf R^+)}.
\end{align*} Thus,  Gronwall's inequality implies that\footnote{Indeed, for $\varkappa(t):=\Vert v w_{N}^\theta\Vert^2_{\mathbf{L}^{2}(\mathbf R^+)},$ we have proved $\varkappa'(t)\lesssim C(t)+\varkappa(t)+\| {v}  \|^2_{\textbf{H}^1(\mathbf R^+)},$ from which we deduce that $\varkappa\in C^1([0,T]),$ and by taking $M=\sup_{0\leq t\leq T}\varkappa(t),$ we have $\varkappa(t)\leq C(t)+M+\| {v}  \|^2_{\textbf{H}^1(\mathbf R^+)}.$  After of applying Gronwall's inequality  we obtain $\varkappa(t)\lesssim_{T} Ce^{\int_{0}^{T}M}=Ce^{MT},$ for all $T>0,$ with $C\equiv C\left(T,\| {v}  \|^2_{\textbf{H}^1(\mathbf R^+)}\right).$  }  we have
\begin{equation*}
    \sup_{t\in [0,T]}\Vert w^{\theta}_N {v}(t,\cdot)  \Vert^2_{L^2(\mathbf{R}^{+})}\leq C e^{MT},
\end{equation*} where $M=\sup_{t\in [0,T]}\Vert \langle x\rangle v(t,\cdot)\Vert^2_{\textbf{L}^{2}(\mathbf R^+)},$ $C$ is a constant depending on both, the initial and the boundary condition. This proves that $v\in \mathbf L^\infty([0,T], \mathbf{L}^{2,1}(\mathbf R^+))$ and the previous estimate \eqref{boundforz(t)} implies that  $u\in \mathbf L^\infty([0,T], \mathbf{L}^{2,1}(\mathbf R^+)).$
\end{proof}

\textbf{Acknowledgements:} {
 D. C. was supported by the FWO Odysseus 1 grant G.0H94.18N: Analysis and Partial Differential Equations, L. E.  was supported by  Gran Sasso Science Institute.
}

\end{document}